\documentclass{amsart}

\usepackage{amssymb}
\usepackage[dvips]{epsfig}
\usepackage{graphicx}
\usepackage{color}

\newcommand{\R}{{\mathbb R}}

\newtheorem{theorem}{Theorem}[section]
\newtheorem{lemma}[theorem]{Lemma}
\newtheorem{proposition}[theorem]{Proposition}
\newtheorem{corollary}[theorem]{Corollary}

\theoremstyle{definition}
\newtheorem{definition}[theorem]{Definition}

\newtheorem*{merci}{Acknowledgements}

\theoremstyle{remark}
\newtheorem{remark}{Remark}[section]
\def\R{{\mathbb R}}

\numberwithin{equation}{section}

\begin{document}
\title[Dispersive perturbations of Burgers equation]{Dispersive perturbations of Burgers and hyperbolic equations I : local theory }
\author[F. Linares]{Felipe Linares}
\address{ IMPA\\ Estrada Dona Castorina 110\\ Rio de Janeiro 22460-320, RJ Brasil}
\email{ linares@impa.br}

\author[D. Pilod]{Didier Pilod}
\address{Instituto de Matem\' atica, Universidade Federal do Rio de Janeiro, Caixa Postal 68530 CEP 21941-97, Rio de Janeiro, RJ Brasil}
\email{didier@im.ufrj.br}
\author[J.-C. Saut]{Jean-Claude Saut}
\address{Laboratoire de Math\' ematiques, UMR 8628,\\
Universit\' e Paris-Sud et CNRS,\\ 91405 Orsay, France}
\email{jean-claude.saut@math.u-psud.fr}

\maketitle
\begin{abstract}
The aim of this paper is to show how a weakly dispersive perturbation of the inviscid Burgers equation improve (enlarge) the space of resolution of the local Cauchy problem.
More generally we will review several problems arising from weak dispersive perturbations of nonlinear hyperbolic equations or systems.
\end{abstract}

\section{Introduction}
This paper is the first of a series on the Cauchy problem for dispersive perturbations of nonlinear hyperbolic equations or systems.
Our motivation is to study the influence of dispersion on the space of resolution, on the lifespan and on the dynamics of solutions to the Cauchy problem for \lq\lq weak\rq\rq \ dispersive perturbations of hyperbolic quasilinear equations or systems, as for instance the  Boussinesq  systems for surface water waves.

 In the present paper we will focus on   the model equation (which was introduced by  Whitham  \cite{W} for a special  choice of  the kernel $k$, see below):
\begin{equation}\label{Whit}
u_t+uu_x+\int_{-\infty}^{\infty}k(x-y)u_x(y,t)dy=0.
\end{equation}
This equation can also be written on the form
\begin{equation}\label{Whibis}
u_t+uu_x-Lu_x=0,
\end{equation}
where the Fourier multiplier operator $L$ is defined by 
$$\widehat{Lf}(\xi)=p(\xi)\hat{f}(\xi),$$
where $p=\hat{k}.$

Precise assumptions on $k$ (resp. $p$) will be made later on.  In the original Whitham equation, the kernel $k$ was given by 
\begin{equation}\label{tanh}
k(x)=\frac{1}{2\pi}\int_\R \left( \frac{\tanh \xi}{\xi} \right)^{1/2} e^{ix\xi} d\xi,
\end{equation}
that is $p(\xi)=\left( \frac{\tanh \xi}{\xi} \right)^{1/2} .$

The dispersion is in this case that of the finite depth surface water waves without surface tension.

The general idea is to investigate the \lq\lq fight\rq\rq \,  between nonlinearity and dispersion. Usually people attack this problem by fixing the dispersion ({\it eg} that of the KdV equation) and varying the nonlinearity (say $u^p u_x$ in the context of generalized KdV).

Our viewpoint, which is probably more physically relevant, is to fix the quadratic, nonlinearity ({\it eg} $uu_x$) and to vary (lower) the dispersion. In fact in many problems arising from Physics or Continuum Mechanics the nonlinearity is quadratic, with terms like $(u\cdot \nabla) u$ and the dispersion is in some sense weak. In particular the dispersion is not strong enough for yielding the dispersive estimates that allows to solve the Cauchy problem in relatively large functional classes (like the KdV or Benjamin-Ono equation in particular), down to the energy level for instance.\footnote{And thus obtaining {\it global well-posedness} from the conservation laws.}

Two basic issues  can be addressed, a third one will be presented in the final section.

 \vspace{0.5cm}
 
 {\bf 1.} Which amount of dispersion prevents the hyperbolic ({\it ie} by shock formation) blow-up of the underlying hyperbolic quasilinear equation or system. This question has been apparently raised for the first time by Whitham (see \cite{W}) for the Whitham equation \eqref {Whit}.  A physicist's view of that problem is displayed in \cite{KZ} where it is claimed that the collapse of gradients (wave breaking) is prevented when 
 $p(\xi)=|\xi|^{\alpha},$ $\alpha >0$ \lq\lq by comparison of linear and nonlinear terms\rq\rq.
 
 A typical result (see \cite{NS},  \cite{CE}) suggest that  for {\it not too  dispersive Whitham type equations} that is for instance when $p(\xi)=|\xi|^{\alpha}, \; -1<\alpha\leq 0,$ \eqref{Whit} presents a blow-up phenomenum. This has been proved for  Whitham type equations, with a regular kernel $k$ satisfying
 \begin{equation}\label{k}
 k\in C(\R)\cap L^1(\R),\;\text{symmetric and monotonically decreasing on}\; \R_+,
 \end{equation}
by Naumkin and Shishmarev \cite {NS} and by Constantin and Escher \cite{CE}, without an unnecessary hypothesis made in \cite{NS}. The blow-up is obtained for initial data which are sufficiently asymmetric. More precisely :
 
 \begin{theorem}\label{CE} \cite{CE}
 Let $u_0\in H^{\infty}(\R)$ be such that
 $$\inf_{x\in\R}|u'_0(x)|+\sup_{x\in \R}|u'_0(x)|\leq -2k(0).$$
 Then the corresponding solution of \eqref{Whit}  undergoes a wave breaking phenomena, that is there exists $T=T(u_0)>0$ with 
$$\sup_{(x,t)\in [0,T)\times \R}|u(x,t)|<\infty,\text{while}\; \sup_{x\in \R}|u_x(t,x)|\to \infty\;\text{as}\;t\to T.$$
 \end{theorem}
 
\vspace{0.5cm}
 The previous result does not include the case of the Whitham equation \eqref{Whit} with kernel given by \eqref{tanh} since then $k(0)=\infty,$ but it is claimed in \cite{CE} that the method of proof adapts to more general kernels. 
 
 This has been proven recently by Castro, Cordoba and Gancedo \cite{CCG} for 
the equation
\begin{equation}\label {CCG}
u_t+uu_x+D^\beta \mathcal Hu=0,
\end{equation}
where $\mathcal H$ is the Hilbert transform and $D^{\beta}$ is the Riesz potential of order $-\beta$, \textit{i.e.}    $D^{\beta}$ is defined via Fourier transform by
\begin{equation} \label{Riesz}
\widehat{D^\beta f}(\xi)= |\xi|^\beta \hat{f}(\xi),
\end{equation} 
for any $\beta \in \mathbb R$. It is established  in \cite {CCG} (see also \cite{VMH} for the case $\beta=\frac{1}{2}$) that for $0\leq\beta<1,$ there exist initial data $u_0\in L^2(\R)\cap C^{1+\delta}(\R),\; 0<\delta<1,$ and $T(u_0)$ such that the corresponding solution $u$ of \eqref{CCG} satisfies 

$$\lim_{t\to T} \|u(\cdot,t)\|_{C^{1+\delta}(\R)}=+\infty.$$

 \vspace{0.3cm}
This rules out the case $-1<\alpha <0$ in our notation.  As observed in \cite{LS}, the proof in \cite{CCG} extends easily to non pure power dispersions, such as \eqref{tanh} and thus to the Whitham equation \eqref{Whit}. Note however that it is not clear whether or not the blow-up displayed in the aforementioned papers is shocklike. The solution is proven (by contradiction) to blow-up in a $C^{1+\delta}$ norm and the sup norm of the solution and of its derivative might blow-up at the same time.


 The case $0<\alpha<1$ is much more delicate (see the discussion in the final Section).

 \begin{remark}
 Similar issues have been addressed in \cite{KNS} for the Burgers equation with fractionary dissipation.
 \end{remark}
{\bf 2.} Investigate the influence of the dispersive term on the theory of the local well-posedness of the Cauchy problem associated to the general \lq\lq dispersive nonlinear hyperbolic system\rq\rq \ \eqref{basic}. Recall that, for the underlying hyperbolic system (that is when $\mathcal L =0$ in \eqref{basic} below) assumed to be symmetrizable, the Cauchy problem is locally well-posed for data in the Sobolev space $H^s(\R^n)$ for any $s>\frac{n}{2}+1.$

The question is then to look to which extent the presence of $\mathcal L$ can lower the value of $s$. This issue is  well understood for {\it scalar} equations with a relatively high dispersion, as the Korteweg-de Vries, Benjamin-Ono, etc,... equations,  much less for equations or systems with a {\it weak} dispersive part. 

Again, we will focus in the present paper on the scalar equation \eqref{Whit} on its form \eqref{Whibis} that is
\begin{equation} \label{dispBurgers}
\partial_tu- D^{\alpha}\partial_xu+ u\partial_xu=0,
\end{equation}
where $x, t \in \mathbb R$ and $D^{\alpha}$ is the Riesz potential of order $-\alpha$ defined in \eqref{Riesz}.
When $\alpha=1$, respectively $\alpha=2$, equation \eqref{dispBurgers} corresponds to the well-known Benjamin-Ono and respectively Korteweg-de Vries equations. This equation has been extensively studied for $1 \leq \alpha \leq 2$ (see \cite{FLP1}  and the references therein). In the following we will consider the less dispersive case $0<\alpha<1$. The case $\alpha =\frac{1}{2}$ is somewhat reminiscent of the linear dispersion of finite depth water waves with surface tension that have phase velocity (in dimension one and two, where $\hat{{\bf k}}$ is a unit vector) which writes in dimension one or two

\begin{equation}\label{WW}
{\bf c}({\bf k})=\frac{\omega({\bf k})}{|{\bf k}|}\hat{{\bf k}}=
g^{\frac{1}{2}} \left(\frac{\tanh(|{\bf k}|h_0)}{|{\bf
k}|}\right)^{\frac{1}{2}} \left(1+\frac{T}{\rho g}|{\bf
k}|^2\right)^{\frac{1}{2}}\hat{{\bf k}},
\end{equation}

\vspace{0.5cm}
In the case $\alpha=0$, equation \eqref{dispBurgers} becomes  the original Burgers equation 
\begin{equation} \label{Burgers}
\partial_t\widetilde{u}= \widetilde{u}\partial_x\widetilde{u},
\end{equation}
by performing the natural change of variable $\widetilde{u}(x,t)=u(x- t,t)$,  while the case $\alpha=-1$ corresponds to the Burgers-Hilbert equation 
\begin{equation} \label{Burgers-Hilbert} 
\partial_tu+\mathcal{H}u= u\partial_xu,
\end{equation}
where $\mathcal{H}$ denotes the Hilbert transform. Equation \eqref{Burgers-Hilbert} has been studied in \cite{CCG,HI}.

The following quantities are conserved by the flow associated to \eqref{dispBurgers},
\begin{equation} \label{M}
M(u)=\int_{\mathbb R}u^2(x,t)dx,
\end{equation}
and 
\begin{equation} \label{H}
H(u)=\int_{\mathbb R}\big( \frac{1}{2} |D^{\frac{\alpha}2}u(x,t)|^2-\frac{1}{6}u^3(x,t)\big) dx.
\end{equation}

Note that by the Sobolev embedding $H^{\frac{1}{6}}(\R)\hookrightarrow L^3(\R)$,  $H(u)$ is well-defined when $\alpha \geq \frac{1}{3}.$
Moreover, equation \eqref{dispBurgers} is invariant under the scaling transformation 
\begin{displaymath} 
u_{\lambda} (x,t)=\lambda^{\alpha}u(\lambda x,\lambda^{\alpha+1}t),
\end{displaymath}
for any positive number $\lambda$. A straightforward computation shows that $\|u_{\lambda}\|_{\dot{H}^s}=\lambda^{s+\alpha-\frac{1}{2}}\|u_{\lambda}\|_{\dot{H}^s}$, and thus the critical index corresponding to \eqref{dispBurgers} is $s_{\alpha}=\frac{1}{2}-\alpha$. In particular, equation \eqref{dispBurgers} is $L^2$-critical for $\alpha=\frac{1}{2}$. 

By using standard compactness methods, one can prove that the Cauchy problem associated to \eqref{dispBurgers} is locally well-posed in $H^s(\mathbb R)$ for $s>\frac{3}{2}$. 

Moreover, interpolation arguments (see \cite{S}) or the following Gagliardo-Nirenberg inequality, (see for example the appendix in \cite{CCG}),
\begin{displaymath} 
\|u\|_{L^3} \lesssim \|u\|_{L^2}^{\frac{3\alpha -1}{3\alpha}}\|D^{\frac{\alpha}{2}}u\|_{L^2}^{\frac{1}{3\alpha}}, \quad \alpha \geq \frac{1}{3},
\end{displaymath} 
combined with the conserved quantities $M$ and $H$ defined in \eqref{M} and \eqref{H} implies the existence of global weak solution in the energy space $H^{\frac{\alpha}{2}}(\mathbb R)$ as soon as $\alpha>\frac{1}{2}$ and for small data in $H^{\frac{1}{4}}(\mathbb R)$ when $\alpha=\frac{1}{2}$ (see \cite {S}). More precisely\footnote{We recall that we excludes the value $\alpha=1$ which corresponds to the Benjamin-Ono equation for which much more complete results are known.}:

 \begin{theorem} \label{triv}
Let $\frac {1}{2}<\alpha<1$ and $u_0\in H^{\frac{\alpha}{2}}(\R)$. Then \eqref{dispBurgers} possesses a global weak solution in $L^{\infty}([0,T];H^{\frac{\alpha}{2}}(\R))$ with initial data $u_0.$ The same result holds when $\alpha=\frac {1}{2}$ provided $\|u_0\|_{L^2}$ is small enough.
\end{theorem}

Moreover, it was established in \cite{GV} that a Kato type local smoothing property holds, implying global existence of weak $L^2$ solutions :

\begin{theorem}\label {GV}
Let $\frac {1}{2}<\alpha<1$ and $u_0\in L^2(\R)$. Then \eqref{dispBurgers} possesses a global weak solution in  $L^{\infty}([0,\infty);L^2(\R))\cap l^{\infty}L^2_{\text{loc}}(\R;H^{\frac{\alpha}2}_{\text{loc}}(\R))$ with initial data $u_0.$ 

\end{theorem}

 However, the case $0<\alpha <\frac{1}{2}$ is more delicate and the previous results are not known to hold. In particular  the Hamiltonian $H$ together with the $L^2$ norm do not control the $H^{\frac{\alpha}{2}}(\mathbb R)$ norm anymore. Note that the Hamiltonian does not make sense when $0<\alpha<\frac{1}{3}$.

\vspace{0.5cm}

The main result of this paper  establishes that the space of resolution of the local Cauchy problem enlarges with $\alpha$. More precisely we will prove

\begin{theorem} \label{maintheo}
Let $0<\alpha<1$. Define
$s(\alpha)=\frac32-\frac{3\alpha}8$ and assume that $s >
s(\alpha)$. Then, for every $u_0 \in H^s(\mathbb R)$, there exists a
positive time $T =T(\|u_0\|_{H^s})$ (which can be chosen as a nonincreasing function of its argument), and a unique solution $u$ to
\eqref{dispBurgers} satisfying $u(\cdot,0)=u_0$ such that
\begin{equation} \label{maintheo1}
u \in  C([0,T]:H^s(\mathbb R)) \quad \text{and} \quad \partial_xu \in L^1([0,T]:L^{\infty}(\mathbb R))  .
\end{equation}
Moreover, for any $0<T'<T$, there exists a neighborhood
$\mathcal{U}$ of $u_0$ in $H^s(\mathbb R)$ such that the flow map
data-solution
\begin{equation} \label{maintheo2}
S^s_{T'}: \mathcal{U} \longrightarrow C([0,T'];H^s(\mathbb R)) , \
u_0 \longmapsto u,
\end{equation}
is continuous.
\end{theorem}

\begin{remark} \label{IP}Theorem \ref{maintheo} fails in the case $\alpha=0$. Indeed, it is a classical result that the  
IVP associated to Burgers equation \eqref{Burgers} is ill-posed in $H^{\frac32}(\mathbb R)$. For the convenience of the reader, we sketch briefly the proof of this fact. \\
Let $u_0 \in H^{\frac32}(\mathbb R) \setminus W^{1,\infty}(\mathbb R)$ be such that $u_0'(x) \underset{x \to 0}{\longrightarrow} -\infty$ and $u_0 \in C^{\infty}(\mathbb R \setminus \{0\})$. We approximate $u_0$ by a regularizing sequence $\{ u_{0,\epsilon}\} \subset C^{\infty}(\mathbb R) \cap H^2(\mathbb R)$ such that \begin{equation}\label{Billposed}
\|u_{0,\epsilon}-u_0\|_{H^{\frac32}} \underset{\epsilon \to 0}{\longrightarrow}0 \quad \text{and} \quad u_{0,\epsilon}'(0) \to -\infty \quad \text{as}\quad \epsilon \to 0.
\end{equation}
Assume that the IVP associated to \eqref{Burgers} is well-posed in $H^{\frac32}(\mathbb R)$. Then there exists a positive time $T=T_{\frac32}(u_0)$ and a solution $u\in C([0,T]:H^{\frac32}(\R))$ of \eqref{Burgers} emanating from $u_0$. Let us prove first that we have the {\it persistency property}, that is, assuming furthermore  that $u_0\in H^s(\R),$ for some $s>\frac{3}{2},$ that the corresponding solution is defined in $C([0,T']: H^s(\R)),$ with $T'>T.$ We follow an argument in \cite{S}. For any $\eta>0$ such that $\frac{3}{2}<\frac{3}{2}+\eta<s,$ we have $\|u_x\|_{L^\infty}\leq \frac{C}{\sqrt \eta}\|u\|_{H^{3/2+\eta}}$ and the interpolation inequality $\|u\|_{H^{\frac32+\eta}}\leq C \|u\|^{1-\theta}_{H^{\frac32}}\|u\|^\theta_{H^s},$ where $\theta=\frac{\eta}{s-\frac32}.$ By the classical $H^s, s>\frac{3}{2},$ theory of the Burgers equation, one has the energy estimate on $[0,T]$
$$\frac{d}{dt}\|u\|_{H^s}^2\leq C\|u_x\|_{L^{\infty}}\|u\|^2_{H^s}.$$
Consequently, $\|u(t)\|_{H^s}^2$ is majorized by the solution $y(t)$ of the differential equation
$$y'(t)=\frac{C}{\sqrt \eta}y^{1+\theta/2}$$
on its maximal time of existence $[0, T(\eta)].$ One easily finds that $y^{\theta/2}(t)=\frac{2y_0^{\theta/2}}{2-\frac{\sqrt \eta}{s-3/2}Cy_0^{\theta/2}t},$ proving that $T(\eta)\to+\infty$ as $\eta \to 0,$ so that $T'>T.$
\\Coming back to the approximate sequence $\{ u_{0,\epsilon}\},$ we denote by $u_{\epsilon}$ the solution associated to $u_{0,\epsilon}$ it follows from \eqref{Billposed} and the standard theory of the Burgers equation that its associated time of existence $T_{\epsilon,s}$ in $H^s$ for any  $s>\frac{3}{2}$ satisfies 
\begin{displaymath}
T_{\epsilon,s}=-\frac{1}{\inf_{ \R} u'_{0,\epsilon}} \to 0 \quad \text{as}\quad \epsilon \to 0.
\end{displaymath}
By the previous consideration, the existence time $T_{\epsilon,\frac32}$ in $H^{\frac32}(\R)$ satisfies\\ $T_{\epsilon,\frac32}\leq T_{\epsilon,s}.$
Letting $\epsilon \to 0,$ the continuity of the flow map in $H^{\frac32}(\mathbb R)$ would imply $T=0,$ a contradiction.
\end{remark}

\begin{remark} 
In the case $\alpha=1$ in Theorem \ref{maintheo}, we get $s(1)=\frac98$, which corresponds to Kenig, Koenig's result for BO \cite{KK}.
\end{remark}

\begin{remark} Of course, the problem to prove well-posedness in $H^{\frac{\alpha}2}(\mathbb
R)$ in the case $\frac12 \le \alpha<1$, which would imply global
well-posedness by using the conserved quantities \eqref{M} and
\eqref{H}, is still open. This conjecture is supported by the numerical simulations in \cite{KS} that suggest that the solution is global in this case. The use of the techniques in \cite{HIKK}  might be useful to lower the value of $s$. Observe that the value $\alpha= 1/2$ is the {\it $L^2$ critical exponent}.
\end{remark}
\begin{remark}

It has been proven in \cite{MST}  that, for $0<\alpha <2$ the Cauchy problem is $C^2$- ill-posed\footnote{That is that the flow map cannot be $C^2.$} for initial data in any Sobolev spaces $H^s(\mathbb R)$, $s \in \mathbb R$, and in particular that the Cauchy problem cannot be solved by a Picard iterative scheme implemented on the Duhamel formulation. 

 On the other hand, it is
well-known that one can still prove local
well-posedness\footnote{Assuming only the continuity of the flow}
for equation \eqref{dispBurgers} below $H^{\frac32+}(\mathbb R)$
when $ \alpha  \ge 1$. Actually, the Benjamin-Ono equation
(corresponding to $\alpha=1$) is well-posed in $L^2(\mathbb R)$
\cite{IK,MP} as well as  equation \eqref{Burgers-Hilbert} when
$1<\alpha<2$ \cite{HIKK} (see also \cite{Guo} for former results).
The question to know whether the same occurs in the
case $0<\alpha<1$ seems to be still open.
\end{remark}

\begin{remark}
Theorem \ref{maintheo}   extends easily by perturbation  to some non  pure power dispersions. For instance, in the case of \eqref{WW}, it suffices to observe that
$$(1+\xi^2)^{1/2}\left(\frac{\tanh |\xi|}{|\xi|}\right)^{1/2}=|\xi|^{1/2}+R(|\xi|),$$ 
where $|R(|\xi|)|\leq |\xi|^{-3/2}$ for large $|\xi|.$
\end{remark}

\begin{remark}
 One could wonder about  the existence of global solutions with small initial data. This was solved in \cite{SSS} when $\alpha\geq 1$
but the case $\alpha<1$ seems to be open. 
\end{remark}

\begin{remark}
It has been proven in \cite{SSS} that the fundamental solution $G_{\alpha}$ of \eqref{dispBurgers} can be written as
$$G_{\alpha}(x,t)=t^{-1/(\alpha +1)}A\left(\frac{x}{t^{1/(\alpha +1)}}\right),$$
where $A$ satisfies the following anisotropic behavior at infinity :
$$
A(z)\underset{z \to +\infty}{\sim}\frac{C}{|z|^{\alpha +2}},\quad \text{and}
\quad A(z)\underset{z \to -\infty}{\sim} C|z|^{\frac{1-\alpha}{2\alpha}} \cos\left(\alpha \left(\frac{|z|}{\alpha+1}\right)+\frac{\pi}{4}\right)\,.
$$
This suggests the possibility of existence of global weak solutions with initial data in a $L^2$ space with anisotropic weight as it is the case for the KdV equation (see \cite{Ka}).
\end{remark}

We now discuss the main ingredients in the proof of Theorem \ref{maintheo}. Since we cannot prove Theorem \ref{maintheo} by a contraction method as explained above, we  use a compactness argument. Standard energy estimates, the Kato-Ponce commutator estimate and Gronwall's inequality provide the following bound for smooth solutions 
\begin{displaymath} 
\|u\|_{L^{\infty}_TH^s_x} \le c \|u_0\|_{H^s_x}e^{c\int_0^T\|\partial_xu\|_{L^{\infty}_x}dt}.
\end{displaymath}  
Therefore, it is enough to control $\|\partial_xu\|_{L^1_TL^{\infty}_x}$ at the $H^s$-level to obtain our \textit{a priori} estimates.

Note that the classical Strichartz estimate for the free group $e^{tD^{\alpha}\partial_x}$ associated to the linear part of \eqref{dispBurgers}, and derived by Kenig, Ponce and Vega in \cite{KPV1}, induces a loss of $\frac{1-\alpha}4$ derivatives in $L^{\infty}$, since we are in the case $0<\alpha<1$ (see Remark \ref{RemarkStrichartz} below). Then, we need to use a refined version of this Strichartz estimate, derived by chopping the time interval in small pieces whose length depends on the spatial frequency of the function (see Proposition \ref{refinedStrichartz} below). This estimate was first established by Kenig and Koenig \cite{KK}  (based on previous ideas of Koch and Tzevtkov \cite{KoTz}) in the Benjamin-Ono context (when $\alpha=1$) .  

We also use  a maximal function estimate for $e^{tD^{\alpha}\partial_x}$  in the case $0<\alpha<1$, which follows directly from the arguments of Kenig, Ponce and Vega \cite{KPV2}. Moreover to complete our argument, we need a local smoothing effect for the solutions of the nonlinear equation \eqref{dispBurgers}, which is based on series expansions and remainder estimates for commutator of the type $[D^{\alpha}\partial_x,u]$ derived by Ginibre and Velo \cite{GV}.

All those estimates allow us to obtain the desired \textit{a priori} bound for $\|\partial_xu\|_{L^1_TL^{\infty}_x}$ at the $H^s$-level, when $s>s(\alpha)=\frac32-\frac{3\alpha}8$, \textit{via} a recursive argument.  Finally, we conclude the proof of Theorem \ref{maintheo}, by applying the same method to the differences of two solutions of \eqref{dispBurgers} and by using the Bona-Smith argument \cite{BS}.

The plan of the paper is as follows. The next section is devoted to the proof of Theorem \ref{maintheo}. We then prove a ill-posedness result for \eqref{dispBurgers}, namely that the flow map cannot be uniformly continuous when $\frac{1}{3}<\alpha<\frac{1}{2}$, based on the existence of solitary waves the theory of which is briefly surveyed  in the last section where we present also some  open questions concerning blow-up or long time existence that will be considered in subsequent  works and comment briefly on the BBM version of the dispersive Burgers equation \eqref{dispBurgers}.

\subsection*{Notations}

The following notations will be used throughout this article: $D^s=(-\Delta)^{\frac{s}2}$ and $J^s=(I-\Delta)^{\frac{s}2}$ denote the Riesz and Bessel potentials of order $-s$, respectively. $\mathcal{H}$ denotes the Hilbert transform. Observe then that $D^1=\mathcal{H}\partial_x$.

For $1 \le p \le \infty$, $L^p(\mathbb R)$ is the usual Lebesgue space with the norm $\|\cdot \|_{L^p}$, and for $s \in \mathbb R$, the Sobolev spaces $H^s(\mathbb R)$ is defined via its usual norm $\|\phi \|_{H^s}= \|J^s \phi\|_{L^2}$.

Let $f=f(x,t)$ be a function defined for $x \in
\mathbb R$ and $t$ in the time interval $[0,T]$, with $T>0$ or in the whole line $\mathbb R$. Then if $X$
is one of the spaces defined above, we define the spaces $L^p_TX_x$ and
$L^p_tX_x$ by the norms
\begin{displaymath}
\|f\|_{L^p_TX_x} =\Big(
\int_{0}^T\|f(\cdot,t)\|_{X}^pdt\Big)^{\frac1p} \quad \text{and} \quad
\|f\|_{L^p_tX_x} =\Big( \int_{\mathbb R}\|f(\cdot,t)\|_{X}^pdt\Big)^{\frac1p},
\end{displaymath}
when $1 \le p < \infty$, with the natural modifications for $p=\infty$. Moreover, we use similar definitions for the spaces $L^q_xL^p_t$ and $L^q_xL^p_T$, with $1 \le p, \ q \le \infty$.

Finally, we say that $A \lesssim B$ if there exists a constant $c>0$ such that $A \le cB$ (it will be clear from the context what parameters $c$ may depend on).

\section{Proof of Theorem \ref{maintheo}} We start by proving various dispersive estimates.

\subsection{Linear estimates, energy estimates and local smoothing effect}
In this section, we consider the linear IVP associated to \eqref{dispBurgers}
\begin{equation} \label{dispBurgerslinear}
\left\{\begin{array}{ll}
\partial_tu-D^{\alpha}\partial_xu=0 \\ 
u(x,0)=u_0(x),
\end{array}
\right.
\end{equation}
whose solution is given by the unitary group $e^{tD^{\alpha}\partial_x}$, defined by
\begin{equation} \label{group}
e^{tD^{\alpha}\partial_x}u_0=\mathcal{F}^{-1}\big(e^{it|\xi|^{\alpha}\xi}\mathcal{F}(u_0)\big) \ .
\end{equation}
We will study the properties of $e^{tD^{\alpha}\partial_x}$ in the case where $0<\alpha<1$.

\subsection{Strichartz estimates} The following estimate is obtained as an application of Theorem 2.1 in \cite{KPV1}.
\begin{proposition} \label{Strichartz}
Assume that $0<\alpha<1$. Let $q$ and  $r$ satisfy $\frac2q+\frac1r=\frac12$ with $2 \le q,r \le +\infty$. Then 
\begin{equation} \label{Strichartz1} 
\|e^{tD^{\alpha}\partial_x}D^{\frac{\alpha-1}q}u_0 \|_{L^q_tL^r_x} \lesssim \|u_0\|_{L^2} \ , 
\end{equation}
for all $u_0 \in L^2(\mathbb R)$.
\end{proposition}

\begin{remark} \label{RemarkStrichartz}
In particular, if we choose $(q,r)=(4,\infty)$, then we obtain from \eqref{Strichartz1} a Strichartz estimate with a lost of $(1-\alpha)/4$ derivatives 
\begin{displaymath}  
\|e^{tD^{\alpha}\partial_x}u_0 \|_{L^4_tL^{\infty}_x} \lesssim \|D^{\frac{1-\alpha}4}u_0\|_{L^2} \ . 
\end{displaymath}
\end{remark}

Next, we derive a refined Strichartz estimate for solutions of the nonhomogeneous linear equation
\begin{equation} \label{dispBurgerslinearnonhomog} 
\partial_tu-D^{\alpha}\partial_xu=F \ .
\end{equation}
This estimate generalizes the one derived by Kenig and Koenig in the Benjamin-Ono case $\alpha=1$ (c.f. Proposition 2.8 in \cite{KK}). Note that the proof of Proposition 2.8 in \cite{KK} is based on previous ideas of Koch and Tzvetkov \cite{KoTz}.

\begin{proposition} \label{refinedStrichartz}
Assume that $0<\alpha<1$, $T>0$ and $\delta \ge 0$. Let $u$ be a smooth solution to \eqref{dispBurgerslinearnonhomog} defined on the time interval $[0,T]$. Then, there exist $0<\kappa_1, \ \kappa_2<\frac12$ such that 
\begin{equation} \label{refinedStrichartz1}
\|\partial_xu\|_{L^2_TL^{\infty}_x} \lesssim T^{\kappa_1}\|J^{1+\frac{\delta}4+\frac{1-\alpha}4+\theta}u\|_{L^{\infty}_TL^2_x}+T^{\kappa_2}\|J^{1-\frac{3\delta}4+\frac{1-\alpha}4+\theta}F\|_{L^2_{T,x}}  \ ,
\end{equation}
for any $\theta>0$.
\end{proposition}

\begin{remark} In our analysis, the optimal choice in estimate \eqref{refinedStrichartz1}  corresponds to $\delta=1-\frac{\alpha}2$. Indeed, if we denote $a=1+\frac{\delta}4+\frac{1-\alpha}4+\theta$ and $b=1-\frac{3\delta}4+\frac{1-\alpha}4+\theta$, we should adapt $\delta$ to get $a=b+1-\frac{\alpha}2$, since we need to absorb $1$ derivative appearing in the nonlinear part of \eqref{dispBurgers} and we are able to recover $\frac{\alpha}2$ derivatives by using the smoothing effect associated with solutions of \eqref{dispBurgerslinearnonhomog}. The use of $\delta=1-\frac{\alpha}2$ in estimate \eqref{refinedStrichartz1} provides the optimal regularity $s>s(\alpha)=\frac32-\frac{3\alpha}8$ in Theorem \ref{maintheo}.
\end{remark}

\begin{proof}[Proof of Proposition \ref{refinedStrichartz}] Following the arguments in \cite{KK}, and \cite{KoTz}, we use a nonhomogeneous Littlewood-Paley
decomposition, $u=\sum_Nu_N$ where $u_N=P_Nu$, $N$ is a dyadic number and $P_N$ is a Littlewood-Paley projection around $|\xi| \sim N$. Then, we get from the Sobolev embedding and the Littlewood-Paley theorem
\begin{displaymath}
\|\partial_xu\|_{L^2_TL^{\infty}_x}\lesssim \|J^{\theta'}\partial_xu\|_{L^2_TL^r_x}
\lesssim \Big(\underset{N}{\sum}\|J^{\theta'}\partial_xu_N\|_{L^2_TL^r_x}^2\Big)^{1/2},
\end{displaymath}
whenever $\theta'r>1$. Therefore, it is enough to prove that
\begin{equation} \label{refinedStrichartz2}
\|\partial_xu_N\|_{L^2_TL^{\infty}_x} \lesssim T^{\kappa_1}\|D^{1+\frac{\delta}4+\frac{1-\alpha}4-\frac{\delta+1-\alpha}{2r}}u_N\|_{L^{\infty}_TL^2_x}+T^{\kappa_2}\|D^{1-\frac{3\delta}4+\frac{1-\alpha}4-\frac{\delta+1-\alpha}{2r}}F_N\|_{L^2_{T,x}}
\end{equation}
for any $r>2$ and any dyadic number $N \ge 1$.

In order to prove estimate \eqref{refinedStrichartz2}, we chop out the interval in small intervals of length  $T^{\kappa}N^{-\delta}$ where $\kappa$ is a small positive number to be fixed later. In other words, we have that $[0,T]=\underset{j \in J}{\bigcup}I_j$ where $I_j=[a_j, b_j]$, $|I_j|\thicksim T^{\kappa}N^{-\delta}$ and $\# J\sim T^{1-\kappa}N^{\delta}$. Let $q$ be such that $\dfrac2q+\dfrac1r=\dfrac12$. Then, since $u_N$ satisfies $\partial_tu_N-D^{\alpha}\partial_xu_N=F_N$, we deduce that
\begin{displaymath}
\begin{split}
\|\partial_xu_N\|_{L^2_TL^r_x} &=\Big(\underset{j}{\sum}\|\partial_xu_N\|^2_{L^2_{I_j}L^r_x}\Big)^{1/2}  \\ &
\le (T^{\kappa}N^{-\delta})^{\frac12-\frac1q}\Big(\underset{j}{\sum}\|\partial_xu_N\|^2_{L^q_{I_j}L^r_x}\Big)^{1/2}\\
&\lesssim (T^{\kappa}N^{-\delta})^{\frac12-\frac1q}\Big(\underset{j}{\sum}\|e^{(t-a_j)D^{\alpha}\partial_x}\partial_xu_N(a_j)\|^2_{L^q_{I_j}L^r_x} \\& \quad
+\sum_j\big\|\int\limits_{a_j}^t e^{(t-t')D^{\alpha}\partial_x}\partial_xF_N(t')\,dt'\big\|^2_{L^q_{I_j}L^r_x}\Big)^{1/2} \ .
\end{split}
\end{displaymath}
Therefore, it follows from estimate \eqref{Strichartz1} that
\begin{displaymath}
\begin{split}
\|\partial_xu_N\|_{L^2_TL^{\infty}_x} 
&\lesssim (T^{\kappa}N^{-\delta})^{\frac12-\frac1q}\,\Big\{\Big(\underset{j}{\sum}\|D^{\frac{1-\alpha}q}\partial_xu_N\|^2_{L^{\infty}_TL^2_x}
\Big)^{1/2} \\ & \quad+\Big(\underset{j}{\sum} T^{\kappa}N^{-\delta}\int\limits_{I_j}\|D^{\frac{1-\alpha}q}\partial_xF_N\|_{L^2_x}^2\,dt
\Big)^{1/2}\Big\}\\
&\le (T^{\kappa}N^{-\delta})^{\frac12-\frac1q}(T^{1-\kappa}N^{\delta})^{\frac12}\|D^{1+\frac{1-\alpha}q}u_N\|_{L^{\infty}_TL^2_x} \\ & \quad
+(T^{\kappa}N^{-\delta})^{\frac12-\frac1q}(T^{\kappa}N^{-\delta})^{\frac12}\Big(\int\limits_0^T\|D^{1+\frac{1-\alpha}q}F_N\|_{L^2_x}^2\,dt\Big)^{1/2}\\
&\lesssim T^{\frac12-\frac{\kappa}q}\|D^{1+\frac{1-\alpha}q+\frac{\delta}q}u_N\|_{L^{\infty}_TL^2_x}+T^{\kappa(1-\frac1q)}
\|D^{1+\frac{1-\alpha}q-\delta+\frac{\delta}{q}}F_N\|_{L^2_{T,x}} \ ,
\end{split}
\end{displaymath}
which implies estimate \eqref{refinedStrichartz2} since $\dfrac1q=\dfrac14-\dfrac{1}{2r}$. Thus given $\theta>0$, choosing $\theta'$ and
$r$ such that $\theta'-\frac{\delta+1-\alpha}{2r}<\theta$, $\kappa_1=\frac12-\frac{\kappa}q$, $ \kappa_2=\kappa(1-\frac1q)$, with $\kappa=\frac12$, the lemma follows.
\end{proof}

\subsection{Maximal function estimate}
Arguing as in the proof of Theorem 2.7 in \cite{KPV2}, we get the following maximal function estimate in $L^2$ for the group $e^{tD^{\alpha}\partial_x}$.
\begin{proposition} \label{maximalfunction}
Assume that $0<\alpha<1$. Let $s>\frac12$. Then, we have that 
\begin{equation} \label{maximalfunction1}
\|e^{tD^{\alpha}\partial_x}u_0\|_{L^2_xL^{\infty}_{[-1,1]}} \le \Big(\sum_{j=-\infty}^{+\infty}\sup_{|t| \le 1}\sup_{j \le x < j+1}|e^{tD^{\alpha}\partial_x}u_0(x)|^2 \Big)^{\frac12} \lesssim  \|u_0\|_{H^s} \ ,
\end{equation}
for any $u_0 \in H^s(\mathbb R)$.
\end{proposition}

The key point in the proof of Proposition \ref{maximalfunction} is the analogous of Proposition 2.6 in \cite{KPV2} in the case $0<\alpha<1$.

\begin{lemma} \label{maximallemma}
Assume that $0<\alpha<1$. Let $\psi_k$ be a $C^{\infty}$ function supported in the set $\{\xi \in \mathbb R \ : \ 2^{k-1} \le |\xi| \le 2^{k+1}\}$, where $k \in \mathbb Z_+$.  Then, the function $H_k^{\alpha}$ defined as 
\begin{equation} \label{maximallemma1} 
H^{\alpha}_k(x)= \left\{\begin{array}{ll} 2^k & \text{if} \ |x| \le 1 \\ 2^{\frac{k}2}|x|^{-\frac12} & \text{if} \ 1 \le |x| \le c2^{\alpha k} \\ (1+x^2)^{-1} & \text{if} \ |x|>c2^{\alpha k}\end{array} \right. ,
\end{equation}
satisfies 
\begin{equation} \label{maximallemma2}
\Big|\int_{\mathbb R}e^{i(t|\xi|^{\alpha}\xi+x\xi)}\psi_k(\xi)d\xi \Big| \lesssim H_k^{\alpha}(x) \ ,
\end{equation}
for $|t| \le 2$. Moreover, we have that 
\begin{equation} \label{maximallemma3} 
\sum_{l=-\infty}^{+\infty}H_k^{\alpha}(|l|) \lesssim 2^k .
\end{equation}
Note that  the implicit constants appearing in \eqref{maximallemma2} and \eqref{maximallemma3} do not depend on $t$ or $k$.
\end{lemma}

\begin{proof} The proof of estimate \eqref{maximallemma2} follows exactly as the one of Proposition 2.6 in \cite{KPV2}. Next, we prove estimate \eqref{maximallemma3}. We get from \eqref{maximallemma1} that
\begin{equation} \label{maximallemma4}
\begin{split}
\sum_{l=-\infty}^{+\infty}H_k^{\alpha}(|l|) & = H_k^{\alpha}(0)+\sum_{|l| = 1}^{[c2^{\alpha k}]}H_k^{\alpha}(|l|)+\sum_{|l| \ge [c2^{\alpha k}]+1}H_k^{\alpha}(|l|) \\ & \le 2^k+2\sum_{l = 1}^{[c2^{\alpha k}]}2^{\frac{k}2}l^{-\frac12}+2\sum_{l\ge [c2^{\alpha k}]+1}(1+l^2)^{-1} .
\end{split}
\end{equation}
Now, observe that 
\begin{displaymath} 
\sum_{l = 1}^{[c2^{\alpha k}]}l^{-\frac12} \le \int_0^{[c2^{\alpha k}]}|x|^{-\frac12}dx \lesssim 2^{\frac{\alpha k}2} ,
\end{displaymath}
which implies estimate \eqref{maximallemma3} recalling \eqref{maximallemma4} and the fact that $2^{\frac{(1+\alpha)k}2} \le 2^k$ since $0<\alpha<1$.
\end{proof}

\begin{proof}[Proof of Proposition \ref{maximalfunction}] The proof of Proposition \ref{maximalfunction} is exactly the same as the one of Theorem 2.7 in \cite{KPV2}. The difference in the regularity $s>\frac12$ instead of $s>\frac{1+\alpha}4$ comes from the fact that we use estimate \eqref{maximallemma3} in the last inequality of the estimate at the top of page 333 in \cite{KPV2}. 
\end{proof}

\begin{corollary} \label{maximalcoro} 
Assume that $0<\alpha<1$. Let $s>\frac12$, $\beta>\frac12$ and $T>0$. Then, we have that 
\begin{equation} \label{maximalcoro1}
 \Big(\sum_{j=-\infty}^{+\infty}\sup_{|t| \le T}\sup_{j \le x < j+1}|e^{tD^{\alpha}\partial_x}u_0(x)|^2 \Big)^{\frac12} \lesssim  (1+T)^{\beta}\|u_0\|_{H^s} \ ,
\end{equation}
for any $u_0 \in H^s(\mathbb R)$.
\end{corollary}

\begin{proof} We can always assume that $T > 1$, since the proof of estimate \eqref{maximalcoro1} is a direct consequence of estimate \eqref{maximalfunction1} in the case where $0< T \le 1$. Arguing exactly as in the proof of Corollary 2.8 in \cite{KPV2}, we obtain that 
\begin{equation} \label{maximalcoro2}
\Big(\sum_{j=-\infty}^{+\infty}\sup_{|t| \le T}\sup_{j \le x < j+1}|e^{tD^{\alpha}\partial_x}u_0(x)|^2 \Big)^{\frac12} 
\lesssim T^{\frac{s}{1+\alpha}}\|u_0\|_{H^s}, 
\end{equation}
for any $s>\frac12$. Now fix $\beta>\frac12$. Then, we have that $\beta (1+\alpha)=s_0>\frac12$. In the case where $s<s_0$, estimate \eqref{maximalcoro2} implies directly estimate \eqref{maximalcoro1}. On the other hand if $s_0 \le s$, we apply \eqref{maximalcoro2} with $s_0$, so that the left-hand side of \eqref{maximalcoro1} is bounded by $T^{\beta}\|u_0\|_{H^{s_0}} \le T^{\beta}\|u_0\|_{H^s}$, which also implies the result in this case. 
\end{proof}

\subsection{Energy estimates} 
In this subsection, we prove the energy estimates satisfied by solutions of \eqref{dispBurgers}.

\begin{proposition} \label{EnergyEstimates} 
Assume that $0<\alpha<1$ and $T>0$. Let $u \in C([0,T]:H^{\infty}(\mathbb R))$ be a smooth solution to \eqref{dispBurgers}. Then, 
\begin{equation} \label{EnergyEstimates1}
\|u(\cdot,t)\|_{L^2}=\|u_0\|_{L^2}, 
\end{equation}
for all $t \in [0,T]$. Moreover, if $s>0$ is given, we have 
\begin{equation} \label{EnergyEstimates2}
\|u\|_{L^{\infty}_TH^s_x}\lesssim \|u_0\|_{H^s}e^{c\|\partial_xu\|_{L^1_TL^{\infty}_x}}. 
\end{equation}

\end{proposition}

The proof of estimate \eqref{EnergyEstimates2} relies on the Kato-Ponce commutator estimate \cite{KP} (see also Lemma 2.2 in \cite{Po}).
\begin{lemma} \label{Kato-Ponce}
Let $s>0$, $p,\ p_2, \ p_3 \in (1,\infty)$ and $p_1, \ p_4 \in (1,\infty]$ be such that $\frac1p=\frac1{p_1}+\frac1{p_2}=\frac1{p_3}+\frac1{p_4}$ . Then, 
\begin{equation}  \label{Kato-Ponce1}
\|[J^s,f]g\|_{L^p} \lesssim \|\partial_xf\|_{L^{p_1}}\|J^{s-1}g\|_{L^{p_2}}+\|J^sf\|_{L^{p_3}}\|g\|_{L^{p_4}}.
\end{equation}
and 
\begin{equation}  \label{Kato-Ponce2}
\|J^s(fg)\|_{L^p} \lesssim \|f\|_{L^{p_1}}\|J^{s}g\|_{L^{p_2}}+\|J^sf\|_{L^{p_3}}\|g\|_{L^{p_4}}.
\end{equation}
\end{lemma}

We also state the fractional Leibniz rule proved in \cite{KPV3} which is a refined version of estimate \eqref{Kato-Ponce2} in the case $0<s<1$ and will be needed in the next section.
\begin{lemma} \label{LeibnizRule} 
Let $\sigma =\sigma_1+\sigma_2 \in (0,1)$ with $\sigma_i \in (0,\gamma)$ and $p, \  p_1, \ p_2 \in (1,\infty)$ satisfy
$\frac1p=\frac1{p_1}+\frac1{p_2}$. Then, 
\begin{equation}  \label{LeibnizRule1}
\|D^{\sigma}(fg)-fD^{\sigma}g-gD^{\sigma}f\|_{L^p} \lesssim \|D^{\sigma_1}f\|_{L^{p_1}}\|D^{\sigma_2}g\|_{L^{p_2}}.
\end{equation}
Moreover, the case $\sigma_2=0$, $p_2=\infty$ is also allowed.
\end{lemma}

\begin{proof}[Proof of Proposition \ref{EnergyEstimates}] We obtain identity \eqref{EnergyEstimates1} multiplying equation \eqref{dispBurgers} by $u$, integrating in space and using that the operator $D^{\alpha}\partial_x$ is skew-adjoint. 

To prove estimate \eqref{EnergyEstimates2}, we apply the operator $J^s$ to \eqref{dispBurgers}, multiply by $J^su$ and integrate in space, which gives
\begin{equation} \label{EnergyEstimate3}
\frac12\frac{d}{dt}\int_{\mathbb R}|J^su|^2dx=-\int_{\mathbb R}J^su[J^s,u]\partial_xudx-\int_{\mathbb R}J^su(J^s\partial_xu)udx.
\end{equation}
We use the commutator estimate \eqref{Kato-Ponce1} to treat the first term appearing on the right-hand side of \eqref{EnergyEstimate3} and integrate by part to handle the second one. It follows that 
\begin{equation} \label{EnergyEstimate4}
 \frac12\frac{d}{dt}\|J^su(\cdot,t)\|_{L^2}^2 \lesssim \|\partial_xu(\cdot,t)\|_{L^{\infty}_x}\|J^su(\cdot,t)\|_{L^2_x}^2.
\end{equation}
Therefore, we deduce estimate \eqref{Kato-Ponce1} applying Gronwall's inequality to \eqref{EnergyEstimate4}.
\end{proof}

\subsection{Local smoothing effect}

By using Theorem 4.1 in \cite{KPV1}, we see that the solutions of the linear equation \eqref{dispBurgerslinear} recover  $\alpha/2$ spatial  derivatives locally in space. 
\begin{proposition} \label{localsmoothinglinear}
Assume that $0<\alpha<1$. Then, we have that 
\begin{equation} \label{localsmoothinglinear1}
\|D^{\frac{\alpha}2}e^{tD^{\alpha}\partial_x}u_0\|_{L^{\infty}_xL^2_T} \lesssim \|u_0\|_{L^2},
\end{equation}
for any $u_0 \in L^2(\mathbb R)$.
\end{proposition}

However in our analysis, we will need a nonlinear version of Proposition \ref{localsmoothinglinear}, whose proof uses the original ideas of Kato \cite{Ka}. 
\begin{proposition} \label{localsmoothing}
Let $\chi$ denote a nondecreasing smooth function such that $\text{supp} \, \chi' \subset (-1,2)$ and $\chi_{|_{[0,1]}}=1$. For $j \in \mathbb Z$, we define $\chi_j(\cdot)=\chi(\cdot-j)$.
Let $u \in C([0,T]:H^{\infty}(\mathbb R))$ be a smooth solution of \eqref{dispBurgers} satisfying $u(\cdot,0)=u_0$ with $0<\alpha<1$. Assume also that $s \ge 0$ and $l>\frac12$. Then, 
\begin{equation} \label{localsmoothing1} 
\begin{split}
\Big(\int_0^T\int_{\mathbb R}\big(|D^{s+\frac{\alpha}2}&u(x,t)|^2+|D^{s+\frac{\alpha}2}\mathcal{H}u(x,t)|^2\big)\chi'_j(x)dxdt \Big)^{\frac12}\\& 
\lesssim   \big(1+T+\|\partial_xu\|_{L^1_TL^{\infty}_x} +T\|u\|_{L^{\infty}_TH^l_x}\big)^{\frac12}\|u\|_{L^{\infty}_TH^s_x} .
\end{split}
\end{equation} 
\end{proposition}

The proof of Proposition \ref{localsmoothing} is based on the following identity. 
\begin{lemma} \label{localsmoothinglemma}
Assume $0<\alpha<1$. Let $h \in C^{\infty}(\mathbb R)$ with $h'$ having compact support. Then, 
\begin{equation} \label{localsmoothinglemma1}
\int_{\mathbb R}f(D^{\alpha}\partial_xf)hdx=\frac{\alpha+1}2\int_{\mathbb R} \big( |D^{\frac{\alpha}2}f|^2+|D^{\frac{\alpha}2}\mathcal{H}f|^2\big)h'dx+\int_{\mathbb R}fR_{\alpha}(h)f  ,
\end{equation}
where $\|R_{\alpha}(h)f\|_{L^2} \le c_{\alpha}\|\mathcal{F}(D^{\alpha}h')\|_{L^1}\|f\|_{L^2}$, for any $f \in L^2(\mathbb R)$.
\end{lemma}

\begin{proof}
Plancherel's identity implies that 
\begin{equation} \label{localsmoothinglemma2}
2\int_{\mathbb R}f(D^{\alpha}\partial_xf)hdx=-\int_{\mathbb R}f[D^{\alpha}\partial_x,h]fdx
=\int_{\mathbb R}f[\mathcal{H}D^{\alpha+1},h]fdx ,
\end{equation}
since $D^1=\mathcal{H}\partial_x$. 

On the other hand, we obtain gathering formulas (21), (22), (23) and Proposition 1 in \cite{GV} with $\alpha=2\mu$ and $n=[\mu]=0$ that 
\begin{equation} \label{localsmoothinglemma3}
[\mathcal{H}D^{\alpha+1},h]f=Pf-\mathcal{H}P\mathcal{H}f+R_{\alpha}(h)f
\end{equation}
where 
\begin{displaymath} 
Pf=(\alpha+1)D^{\frac{\alpha}2}(h'D^{\frac{\alpha}2}f) \quad \text{and} \quad \|R_{\alpha}(h)f\|_{L^2}\le c_{\alpha}\|\mathcal{F}(D^{\alpha}h')\|_{L^1} \|f\|_{L^2}.
\end{displaymath}

Therefore, we deduce identity \eqref{localsmoothinglemma1} combining \eqref{localsmoothinglemma2} and \eqref{localsmoothinglemma3}.
\end{proof}

\begin{proof} [Proof of Proposition \ref{localsmoothing}] The proof of Proposition \ref{localsmoothing} follows the lines of the one of Lemma 5.1 in \cite{KPV2}. Apply $D^s$ to \eqref{dispBurgers}, multiply by $D^su\chi_j$ and integrate in space to get 
\begin{equation} \label{localsmoothing2}
\frac12\frac{d}{dt}\int_{\mathbb R}|D^su|^2\chi_jdx-\int_{\mathbb R}D^{s+\alpha}\partial_xuD^su\chi_jdx+\int_{\mathbb R}D^s(u\partial_xu)D^su\chi_jdx=0.
\end{equation}
We use estimate \eqref{Kato-Ponce1} and integration by parts to deal with the third term on the left-hand side of \eqref{localsmoothing2} 
\begin{equation} \label{localsmoothing3} 
\begin{split}
\int_{\mathbb R}D^s(u\partial_xu)D^su\chi_jdx&=-\frac12\int_{\mathbb R}(\partial_xu\chi_j+u\chi_j')|D^su|^2dx+\int_{\mathbb R}[D^s,u]\partial_xuD^su\chi_jdx \\ & \lesssim \big(\|\partial_xu\|_{L^{\infty}_x}+ \|u\|_{L^{\infty}_x}\big)\|u\|_{H^s}^2.
\end{split}
\end{equation}
Therefore, we deduce gathering \eqref{localsmoothinglemma1}, \eqref{localsmoothing2}, \eqref{localsmoothing3} and integrating in time that
\begin{displaymath} 
\begin{split}
\int_0^T\int_{\mathbb R}\big(|D^{s+\frac{\alpha}2}u(x,t)|^2&+|D^{s+\frac{\alpha}2}\mathcal{H}u(x,t)|^2\big)\chi'_j(x)dxdt  \\ & \lesssim \big(1+T+\|\partial_xu\|_{L^1_TL^{\infty}_x} +T\|u\|_{L^{\infty}_{T,x}}\big)\|u\|_{L^{\infty}_TH^s_x}^2,
\end{split}
\end{displaymath}
which implies estimate \eqref{localsmoothing1} by using the Sobolev embedding. 
\end{proof}


The starting point for the proof of Theorem \ref{maintheo} is a well-posedness for smooth solutions obtained in \cite{S} (note that we also need to use the Bona-Smith method \cite{BS} to obtain the continuity of the flow).
\begin{theorem} \label{smoothsol}
Let $0<\alpha<1$ and $s>\frac32$. For any $u_0 \in H^s(\mathbb R)$, there exists a positive time $T=T(\|u_0\|_{H^s})$ and a unique solution to \eqref{dispBurgers} $u \in C([0,T]:H^s(\mathbb R))$ satisfying $u(\cdot,0)=u_0$. Moreover, the map $:u_0 \in H^s(\mathbb R) \mapsto u \in C([0,T]:H^s(\mathbb R))$ is continuous.
\end{theorem}

\subsection{\textit{A priori} estimates for smooth solutions} 

\begin{proposition} \label{apriori}
Assume $0<\alpha<1$ and $s > \frac32-\frac{3\alpha}8$. For any $M>0$, there exists a positive time $\widetilde{T}=\widetilde{T}(M)$ such that for any initial data $u_0 \in H^{\infty}(\mathbb R)$ satisfying $\|u_0\|_{H^s} \le M$, the solution $u$ obtained in Theorem \ref{smoothsol} is defined on $[0,\tilde{T}]$ and satisfies
\begin{equation} \label{apriori1}
\Lambda_T^s(u) \le C_s(\widetilde{T})\|u_0\|_{H^s},
\end{equation}
for all $T \in (0,\widetilde{T}]$, where 
\begin{equation} \label{apriori2}
\Lambda_T^s(u):=\max\Big\{\|u\|_{L^{\infty}_{T}H^s_x}, \|\partial_xu\|_{L^2_{T}L^{\infty}_x}, (1+T)^{-\rho} \Big(\sum_{j=-\infty}^{+\infty}\|u\|_{L^{\infty}([j,j+1)\times[0,T])}^2 \Big)^{\frac12} \Big\} ,
\end{equation}
$\rho>\frac12$ and $C_s(\widetilde{T})$ is a positive constant depending only on $s$ and $\widetilde{T}$. 
\end{proposition}

\begin{proof} Fix $0<\alpha<1$, $s(\alpha)=\frac32-\frac{3\alpha}8<s\le \frac32$ and set $\theta=s-s(\alpha)>0$. Let $u_0 \in H^{\infty}(\mathbb R)$ and $u \in C([0,T^{\star}): H^{\infty}(\mathbb R))$ be the corresponding solution of \eqref{dispBurgers} obtained from Theorem \ref{smoothsol} and defined on its maximal interval of existence $[0,T^{\star})$.

We want to obtain an \textit{a priori} estimate on the quantity $\Lambda_T^s(u)$ defined in \eqref{apriori2}. For $0<T<T^{\star}$, let us define 
\begin{displaymath} 
\lambda_T^s(u)=\|u\|_{L^{\infty}_TH^s_x}, \quad \gamma_T^s(u)=\|\partial_xu\|_{L^2_{T}L^{\infty}_x}
\end{displaymath}
and
\begin{displaymath}
 \mu_T^s(u)= (1+T)^{-\beta}\Big(\sum_{j=-\infty}^{+\infty}\|u\|_{L^{\infty}([j,j+1)\times[0,T])}^2 \Big)^{\frac12}.
\end{displaymath}
First, we rewrite the energy estimate \eqref{EnergyEstimates2} using the above notations as 
\begin{equation} \label{apriori3}
\lambda_T^s(u) \lesssim \|u_0\|_{H^s}e^{cT^{\frac12}\gamma_T^s(u)} .
\end{equation}

To handle $\gamma_T^s(u)$, we use estimate \eqref{refinedStrichartz} with $\delta=1-\frac{\alpha}2$ and $\theta>0$ defined as above and deduce that
\begin{equation} \label{apriori4} 
\begin{split}
\gamma_T^s(u) & \lesssim T^{\kappa_1}\|u\|_{L^{\infty}_TH^s_x}+T^{\kappa_2}\|J^{s-1+\frac{\alpha}2}(u\partial_xu)\|_{L^2_{T,x}} \\ &
\lesssim T^{\kappa_1}\|u\|_{L^{\infty}_TH^s_x}+T^{\kappa_2}\big(\|u\|_{L^{\infty}_TL^2_x}\|\partial_xu\|_{L^2_TL^{\infty}_x}+\|D^{s-1+\frac{\alpha}2}(u\partial_xu)\|_{L^2_{T,x}} \big) ,
\end{split}
\end{equation}
where $\kappa_1$ and $\kappa_2$ are two positive number (lesser than $\frac12$) given by Proposition \ref{refinedStrichartz}. The fractional Leibniz rule stated in Lemma \ref{LeibnizRule} gives that
\begin{equation} \label{apriori5}
\begin{split}
\|D^{s-1+\frac{\alpha}2}(u\partial_xu)\|_{L^2_{T,x}} & \lesssim 
\|uD^{s-1+\frac{\alpha}2}\partial_xu\|_{L^2_{T,x}}+\big\|\|\partial_xu\|_{L^{\infty}_x}\|D^{s-1+\alpha/2}u\|_{L^2_x}\big\|_{L^2_T} \\ & 
\lesssim 
\|uD^{s-1+\frac{\alpha}2}\partial_xu\|_{L^2_{T,x}}+\|\partial_xu\|_{L^2_TL^{\infty}_x}\|u\|_{L^{\infty}_TH^s_x} .
\end{split}
\end{equation}
Now, we deduce from estimate  \eqref{localsmoothing1} that 
\begin{equation} \label{apriori6}
\begin{split}
\|u&D^{s-1+\frac{\alpha}2}\partial_xu\|_{L^2_{T,x}} \\ &=\Big(\sum_{j=-\infty}^{+\infty}\|uD^{s-1+\frac{\alpha}2}\partial_xu\|_{L^2([j,j+1)\times[0,T])}^2 \Big)^{\frac12} \\& 
\le \Big(\sum_{j=-\infty}^{+\infty}\|u\|_{L^{\infty}([j,j+1)\times[0,T])}^2 \Big)^{\frac12}
\sup_{j\in \mathbb Z}\|D^{s+\frac{\alpha}2}\mathcal{H}u\|_{L^2([j,j+1)\times[0,T])} \\ & 
\lesssim (1+T)^{\beta}\mu_T^s(u)\big(1+T+\|\partial_xu\|_{L^1_TL^{\infty}_x} +T\|u\|_{L^{\infty}_TH^s_x}\big)^{\frac12}\|u\|_{L^{\infty}_TH^s_x}. 
\end{split}
\end{equation}
Therefore, we conclude gathering \eqref{apriori4}--\eqref{apriori6} and using estimate \eqref{apriori3} that 
\begin{equation} \label{apriori7}
\begin{split}
&\gamma_T^s(u) \lesssim  \|u_0\|_{H^s}e^{cT^{\frac12}\gamma_T^s(u)} \\ & \times\Big\{ T^{\kappa_1}+T^{\kappa_2}\gamma_T^s(u)+T^{\kappa_2}
\mu_T^s(u)(1+T)^{\beta}\big(1+T(1+\lambda_T^s(u))+T^{\frac12}\gamma_T^s(u)\big)^{\frac12} \Big\}.
\end{split}
\end{equation}

Using that $u$ solves the integral equation 
\begin{displaymath} 
u(t)=e^{tD^{\alpha}\partial_x}u_0-\int_0^te^{(t-t')D^{\alpha}\partial_x}\big(u\partial_xu \big)(t')dt',
\end{displaymath}
we deduce from estimate \eqref{maximalcoro1} that 
\begin{equation} \label{apriori8}
\mu_T^s(u) \lesssim \|u_0\|_{H^{\frac12+\widetilde{\theta}}}+T^{\frac12}\|J^{\frac12+\widetilde{\theta}}(u\partial_xu)\|_{L^2_{T,x}},
\end{equation}
for any $\widetilde{\theta}>0$. Now, we choose $0<\widetilde{\theta} \le \frac{\alpha}8$ , so that $\frac12+\widetilde{\theta} \le s-1+\frac{\alpha}2$. Thus, we have that
\begin{displaymath}
\|J^{\frac12+\widetilde{\theta}}(u\partial_xu)\|_{L^2_{T,x}} \lesssim \|J^{s-1+\frac{\alpha}2}(u\partial_xu)\|_{L^2_{T,x}} .
\end{displaymath}
It follows then from \eqref{apriori8} and arguing as in \eqref{apriori4}--\eqref{apriori6} that
\begin{equation} \label{apriori9}
\begin{split}
&\mu_T^s(u) \lesssim  \|u_0\|_{H^s}e^{cT^{\frac12}\gamma_T^s(u)} \\ & \times\Big\{ 1+T^{\frac12}\gamma_T^s(u)+T^{\frac12}
\mu_T^s(u)(1+T)^{\beta}\big(1+T(1+\lambda_T^s(u))+T^{\frac12}\gamma_T^s(u)\big)^{\frac12} \Big\}.
\end{split}
\end{equation}

Now, observe that $T^{\kappa_2}\gamma_T^s(u)$, $T^{\frac12}\gamma_T^s(u)$, $T^{\kappa_2}\mu_T^s(u)$, $T^{\frac12}\mu_T^s(u)$ and $T\lambda_T^s(u)$ are nondecreasing functions of $T$ which tend to $0$ when $T$ tends to $0$. We define $\widetilde{T}$ such that 
\begin{equation} \label{apriori10} 
\max \big\{\widetilde{T}^{\kappa_2}\gamma_{\widetilde{T}}^s(u),  \widetilde{T}^{\frac12}\gamma_{\widetilde{T}}^s(u), \widetilde{T}^{\kappa_2}\mu_{\widetilde{T}}^s(u), \widetilde{T}^{\frac12}\mu_{\widetilde{T}}^s(u), \widetilde{T}\lambda_{\widetilde{T}}^s(u)\big\} =\frac12.
\end{equation}
Moreover, we can always assume that $\widetilde{T} < T^{\star}$ by choosing (if necessary) a constant smaller than $1/2$ on the right-hand side of \eqref{apriori10}. Therefore, we deduce gathering \eqref{apriori3}, \eqref{apriori7}, \eqref{apriori9} and \eqref{apriori10} that \eqref{apriori1} holds with some positive constant $C_s(\widetilde{T})$ (which can be chosen greater than $1$). 

Finally, we check that $\widetilde{T} \ge c(M)$ where $M$ is positive constant such that $\|u_0\|_{H^s} \le M$. Indeed, since \eqref{apriori10} holds true, we know that one of the terms appearing on the left-hand side of \eqref{apriori10} is equal to $\frac12$. Without loss of generality, we can assume that $\widetilde{T}^{\kappa_2}\gamma_{\widetilde{T}}^s(u)=\frac12$. This combined with \eqref{apriori1} implies that 
\begin{displaymath} 
\frac12 \le \widetilde{T}^{\kappa_2}C_s(\widetilde{T})\|u_0\|_{H^s} \le \widetilde{T}^{\kappa_2}C_s(\widetilde{T})M,
\end{displaymath}
which concludes the proof of Proposition \ref{apriori}.
\end{proof}

\begin{remark} In the $L^2$ subcritical case $\alpha>\frac12$, one can take advantage of the scaling invariance of equation \eqref{dispBurgers} to assume that the initial data is small in $H^s(\mathbb R)$. The arguments of Theorem  1.1 in \cite{KK} could be used to give an easier proof of Proposition \ref{apriori} in this case. 
\end{remark}

\subsection{Uniqueness and $L^2$-Lipschitz bound of the flow} 
Let $u_1$ and $u_2$ be two solutions of \eqref{dispBurgers} in the class \eqref{maintheo1} for some positive $T$, with respective initial data $u_1(\cdot,0)=\varphi_1$ and $u_2(\cdot,0)=\varphi_2$. We define the positive number $K$ by 
\begin{equation} \label{uniqueness1}
K=\max \big\{\|\partial_xu_1\|_{L^1_TL^{\infty}_x}, \|\partial_xu_2\|_{L^1_TL^{\infty}_x} \big\}.
\end{equation}
 
 We set $v=u_1-u_2$. Then $v$ satisfies 
 \begin{equation} \label{uniqueness2} 
 \partial_tv-D^{\alpha}\partial_xv+v\partial_xu_1+u_2\partial_xv=0,
 \end{equation}
with initial data $v(\cdot,0)=\varphi_1-\varphi_2$. We multiply $\eqref{uniqueness2}$ by $v$, integrate in space and integrate by parts to deduce that 
\begin{displaymath} 
\frac{d}{dt}\|v\|_{L^2}^2 \lesssim \big(\|\partial_xu_1\|_{L^{\infty}_x}+\|\partial_xu_2\|_{L^{\infty}_x} \big)\|v(\cdot,t)\|_{L^2}^2, 
\end{displaymath}
for all $t \in [0,T]$.  It follows then from Gronwall's inequality that 
\begin{equation} \label{uniqueness3}
\|v(\cdot,t)\|_{L^2}=\|u_1(\cdot,t)-u_2(\cdot,t)\|_{L^2} \lesssim e^{cK}\|\varphi_1-\varphi_2\|_{L^2},
\end{equation}
for all $t \in [0,T]$. 

Estimate \eqref{uniqueness3} provides the $L^2$-Lipschitz bound of the flow as well as the uniqueness result of Theorem \ref{maintheo} by taking $\varphi_1=\varphi_2$.

\subsection{Existence} 
Assume that $0<\alpha<1$ and $s>\frac32-\frac{3\alpha}8$. Fix an initial datum $u_0 \in H^s(\mathbb R)$. 

We will use the Bona-Smith argument \cite{BS}.  Let $\rho \in \mathcal{S}(\mathbb R)$, $\int\rho
dx=1$, and $\int x^k\rho(x)dx=0, \ k \in \mathbb Z_{+}$, $0 \le k \le [s]+1$. For any $\epsilon>0$, define $\rho_{\epsilon}(x)=\epsilon^{-1}\rho(\epsilon^{-1}x)$. The following lemma, whose proof can be found in \cite{BS}
(see also Proposition 2.1 in \cite{KP2}), gathers the properties of the smoothing operators which will be used in this
subsection.
\begin{lemma} \label{lemma existence1}
Let $s \ge 0$, $\phi \in H^s(\mathbb R)$ and for any $\epsilon>0$,
$\phi_{\epsilon}=\rho_{\epsilon} \ast \phi$. Then,
\begin{equation} \label{lemma existence1.1}
\|\phi_{\epsilon}\|_{H^{s+\nu}} \lesssim \epsilon^{-\nu}
\|\phi\|_{H^s}, \quad \forall \nu \ge0,
\end{equation}
and
\begin{equation} \label{lemma existence1.2}
\|\phi-\phi_{\epsilon}\|_{H^{s-\beta}} \underset{\epsilon
\rightarrow 0}{=} o(\epsilon^{\beta}), \quad \forall \beta \in
[0,s].
\end{equation}
\end{lemma}

Now we regularize the initial datum by letting $u_{0,\epsilon}=\rho_{\epsilon} \ast u_0$. Since $u_{0,\epsilon} \in H^{\infty}(\mathbb R)$, we deduce from Theorem \ref{smoothsol} that for any $\epsilon>0$, there exist a positive time $T_{\epsilon}$ and a unique solution $$u_{\epsilon} \in C([0,T_{\epsilon}];H^{\infty}(\mathbb R)) \quad \text{satisfying} \quad u_{\epsilon}(\cdot,0)=u_{0,\epsilon}.$$ We observe that $\|u_{0,\epsilon}\|_{H^s} \le \|u_0\|_{H^s}$. Thus, it follows from the proof of Proposition \ref{apriori}, that there exists a positive time $T=T(\|u_0\|_{H^s})$ such that the sequence of solutions $\{u_{\epsilon}\}$ can be extended on the time interval $[0,T]$ and satisfies 
\begin{equation} \label{maintheo3} 
\|u_{\epsilon}\|_{L^{\infty}_TH^s_x} \le \Lambda^s_T(u_{\epsilon}) \lesssim \|u_0\|_{H^s} 
\end{equation}
for all $\epsilon>0$.  Moreover, arguing as in the proof of Proposition \ref{apriori} and using estimate \eqref{lemma existence1.1}, we get that
\begin{equation} \label{maintheo4} 
\|D^{s-1}\partial_x^2u_{\epsilon}\|_{L^2_TL^{\infty}_x} \lesssim \|u_{0,\epsilon}\|_{H^{2s}} \lesssim \epsilon^{-s}\|u_0\|_{H^s},
\end{equation}
\begin{equation} \label{maintheo4b}
\|\partial_x^2u_{\epsilon}\|_{L^2_TL^{\infty}_x} \lesssim \|u_{0,\epsilon}\|_{H^{s+1}} \lesssim \epsilon^{-1}\|u_0\|_{H^s}
\end{equation}
and 
\begin{equation} \label{maintheo4bb}
\|D^{s-1+\frac{\alpha}2}\partial_xu_{\epsilon}\|_{L^2_TL^{\infty}_x} \lesssim \|u_{0,\epsilon}\|_{H^{2s-1+\frac{\alpha}2}} \lesssim \epsilon^{-(s-1+\frac{\alpha}2)}\|u_0\|_{H^s}
\end{equation}
for all $\epsilon>0$.

Now we will prove that $\{u_{\epsilon}\}$ is a Cauchy sequence in $C([0,T]:H^s(\mathbb R))$. We set $v=u_{\epsilon}-u_{\epsilon'}$, for $0<\epsilon<\epsilon'$.  Then $v$ satisfies 
\begin{equation} \label{maintheo5}
\partial_tv-D^{\alpha}\partial_xv+v\partial_xu_{\epsilon}+u_{\epsilon'}\partial_xv=0,
\end{equation}
with $v(\cdot,0)=u_{0,\epsilon}-u_{0,\epsilon'}$.
We deduce gathering \eqref{uniqueness3}, \eqref{lemma existence1.2} and \eqref{maintheo3}  that 
\begin{equation} \label{maintheo6} 
\|v\|_{L^{\infty}_TL^2_x} \underset{\epsilon
\rightarrow 0}{=} o(\epsilon^{s})  \quad \text{and} \quad \|v\|_{L^{\infty}_TH^{\sigma}_x} \underset{\epsilon
\rightarrow 0}{=} o(\epsilon^{s-\sigma}),
\end{equation}
for all $0 \le \sigma < s$. It remains to prove the convergence in $C([0,T]:H^s(\mathbb R))$. Note however that the proof in \cite{Po} does not seem to apply here since the regularity is lower than $3/2$.

\begin{proposition} \label{convergence}
Assume $0<\alpha<1$ and $\frac32-\frac{3\alpha}8<s\le \frac32$. Let $v$ be the solution of \eqref{maintheo5}. Then, there exists a time $T_1=T_1(\|u_0\|_{H^s})$ with $0<T_1<T$ such that 
\begin{equation} \label{convergence1} 
\|v\|_{L^{\infty}_{T_1}H^s_x} \le \Gamma_{T_1}^s(v) \underset{\epsilon \to 0}{\longrightarrow} 0,
\end{equation}
where 
\begin{displaymath}
\Gamma_{T}^s(v)=\max \big\{\lambda_T^s(v), \gamma_T^s(v) \big\}
\end{displaymath}
with
\begin{displaymath} 
\lambda_T^s(v)=\|v\|_{L^{\infty}_TH^s_x}, \quad  \text{and} \quad \gamma_T^s(v)=\|\partial_xv\|_{L^2_{T}L^{\infty}_x} \ .
\end{displaymath}
\end{proposition}

\begin{proof} 
First, we deal with $\lambda_T^s(v)$. We apply $J^s$ to equation \eqref{maintheo5} with $u_1=u_{\epsilon}$ and $u_2=u_{\epsilon'}$, multiply the result by $J^sv$ and integrate in space to deduce that 
\begin{equation} \label{convergence2} 
\frac12\frac{d}{dt}\|v\|_{H^s}^2+\int_{\mathbb R}J^s(v\partial_xu_{\epsilon})J^svdx+\int_{\mathbb R}J^s(u_{\epsilon'}\partial_xv)J^svdx=0.
\end{equation}
We treat the third term on the left-hand side of \eqref{convergence2} by using estimate \eqref{Kato-Ponce1} and integrating by parts. It follows that 
\begin{equation} \label{convergence3}
\begin{split}
\int_{\mathbb R}J^s(u_{\epsilon'}\partial_xv)J^svdx&=\int_{\mathbb R}[J^s,u_{\epsilon'}]\partial_xvJ^svdx-\frac12\int_{\mathbb R}\partial_xu_{\epsilon'}J^svJ^svdx \\ & \lesssim \|\partial_xv\|_{L^{\infty}}\|J^su_{\epsilon'}\|_{L^2}\|v\|_{H^s}+\|\partial_xu_{\epsilon'}\|_{L^{\infty}}\|v\|_{H^s}^2 .
\end{split}
\end{equation}
For the second term, we have applied H\" older's inequality  that
\begin{equation} \label{convergence4}
\begin{split}
\int_{\mathbb R}&J^s(v\partial_xu_{\epsilon})J^svdx \\ &\lesssim \big(\|v\partial_xu_{\epsilon}\|_{L^2}+\|D^s(v\partial_xu_{\epsilon})\|_{L^2}\big)\|v\|_{H^s} \\ & \lesssim \big(\|v\|_{L^2}\|\partial_xu_{\epsilon}\|_{L^{\infty}}+\|D^{s-1}(\partial_xv\partial_xu_{\epsilon})\|_{L^2}+\|D^{s-1}(v\partial_x^2u_{\epsilon})\|_{L^2}\big)\|v\|_{H^s}.
\end{split}
\end{equation}
Now we deduce using estimate \eqref{LeibnizRule1} with $\sigma=s-1 \in (0,1)$ that 
\begin{equation} \label{convergence5}
 \|D^{s-1}(\partial_xv\partial_xu_{\epsilon})\|_{L^2} \lesssim \|\partial_xv\|_{L^{\infty}}\|D^{s-1}\partial_xu_{\epsilon}\|_{L^2}+\|D^{s-1}\partial_xv\|_{L^2}\|\partial_xu_{\epsilon}\|_{L^{\infty}}
\end{equation}
and
\begin{equation} \label{convergence6}
\begin{split}
\|D^{s-1}(v\partial_x^2u_{\epsilon})\|_{L^2} & \lesssim 
\|vD^{s-1}\partial_x^2u_{\epsilon}\|_{L^2}+\|D^{s-1}v\|_{L^2}\|\partial_x^2u_{\epsilon}\|_{L^{\infty}} \\ & \lesssim \|v\|_{L^2}\|D^{s-1}\partial_x^2u_{\epsilon}\|_{L^{\infty}}+\|D^{s-1}v\|_{L^2}\|\partial_x^2u_{\epsilon}\|_{L^{\infty}}.
\end{split}
\end{equation}
Hence, we conclude gathering \eqref{convergence2}--\eqref{convergence6} that
\begin{displaymath}
\begin{split}
\frac{d}{dt}\|v\|_{H^s} & \lesssim \big(\|\partial_xu_{\epsilon}\|_{L^{\infty}} +\|\partial_xu_{\epsilon'}\|_{L^{\infty}}\big)\|v\|_{H^s}+\big(\|u_{\epsilon}\|_{H^s} +\|u_{\epsilon'}\|_{H^s}\big)\|\partial_xv\|_{L^{\infty}} \\ & \quad+\|D^{s-1}\partial_x^2u_{\epsilon}\|_{L^{\infty}} \|v\|_{L^2}+\|\partial_x^2u_{\epsilon}\|_{L^{\infty}}\|D^{s-1}v\|_{L^2}.
\end{split}
\end{displaymath}
Thus, Gronwall's inequality yields
\begin{displaymath} 
\lambda_T^s(v) \lesssim\big(\|u_{0,\epsilon}-u_{0,\epsilon'}\|_{H^s}+\int_0^Tf(t)dt \big) e^{cT^{\frac12}(\|\partial_xu_{\epsilon}\|_{L^2_TL^{\infty}_x} +\|\partial_xu_{\epsilon'}\|_{L^2_TL^{\infty}_x})},
\end{displaymath}
where 
\begin{displaymath} 
f(t)=\big(\|u_{\epsilon}\|_{H^s} +\|u_{\epsilon'}\|_{H^s}\big)\|\partial_xv\|_{L^{\infty}}
+\|D^{s-1}\partial_x^2u_{\epsilon}\|_{L^{\infty}} \|v\|_{L^2}+\|\partial_x^2u_{\epsilon}\|_{L^{\infty}}\|D^{s-1}v\|_{L^2}.
\end{displaymath}
Therefore, we get from H\"older's inequality and \eqref{maintheo3} that
\begin{equation} \label{convergence8}
\lambda_T^s(v) \lesssim \big( T^{\frac12}\|u_0\|_{H^s}\gamma_T^s(v)+g_{\epsilon,\epsilon'}\big)e^{cT^{\frac12}\|u_0\|_{H^s}},
\end{equation}
where 
\begin{displaymath} 
g_{\epsilon,\epsilon'}=\|u_{0,\epsilon}-u_{0,\epsilon'}\|_{H^s}+T^{\frac12}\|D^{s-1}\partial_x^2u_{\epsilon}\|_{L^2_TL^{\infty}_x} \|v\|_{L^{\infty}_TL^2_x}+T^{\frac12}\|\partial_x^2u_{\epsilon}\|_{L^2_TL^{\infty}_x}\|D^{s-1}v\|_{L^{\infty}_TL^2_x}
\end{displaymath}
satisfies 
\begin{equation} \label{convergence9}
g_{\epsilon,\epsilon'} \underset{\epsilon,\epsilon' \to 0}{\longrightarrow}0,
\end{equation}
in view of \eqref{maintheo3}--\eqref{maintheo4b} and \eqref{maintheo6}.

To handle $\gamma_T^s(u)$, we use estimate \eqref{refinedStrichartz} as in Proposition \ref{apriori} and deduce that
\begin{equation} \label{convergence10} 
\gamma_T^s(v)   \lesssim T^{\kappa_1}\|v\|_{L^{\infty}_TH^s_x}+T^{\kappa_2}\|J^{s-1+\frac{\alpha}2}(v\partial_xu_{\epsilon})\|_{L^2_{T,x}}+ 
T^{\kappa_2}\|J^{s-1+\frac{\alpha}2}(u_{\epsilon' }\partial_xv)\|_{L^2_{T,x}},
\end{equation}
where $\kappa_1$ and $\kappa_2$ are two positive number (lesser than $\frac12$) given by Proposition \ref{refinedStrichartz}. We deduce from estimate \eqref{LeibnizRule1} that
\begin{equation} \label{convergence11}
\begin{split}
\|J^{s-1+\frac{\alpha}2}(v\partial_xu_{\epsilon})\|_{L^2_{T,x}} & \lesssim 
\|\partial_xu_{\epsilon}\|_{L^2_TL^{\infty}_x}\|v\|_{L^{\infty}_TH^s_x} +\|D^{s-1+\frac{\alpha}2}\partial_xu_{\epsilon}\|_{L^2_TL^{\infty}_x}\|v\|_{L^{\infty}_TL^2_x}.
\end{split}
\end{equation}
By using estimate \eqref{LeibnizRule1} again, we get that
\begin{equation} \label{convergence12}
\|J^{s-1+\frac{\alpha}2}(u_{\epsilon' }\partial_xv)\|_{L^2_{T,x}} \lesssim \|\partial_xv\|_{L^2_TL^{\infty}_x}\|u_{\epsilon'}\|_{L^{\infty}_TH^s_x}+\|u_{\epsilon'}D^{s-1+\frac{\alpha}2}\partial_xv\|_{L^2_{T,x}}.
\end{equation}
Next, we estimate the second term on the right-hand side of \eqref{convergence12} as follows
\begin{equation} \label{convergence13}
\begin{split}
\|u_{\epsilon'}&D^{s-1+\frac{\alpha}2}\partial_xv\|_{L^2_{T,x}} \\ &=\Big(\sum_{j=-\infty}^{+\infty}\|u_{\epsilon'}D^{s-1+\frac{\alpha}2}\partial_xv\|_{L^2([j,j+1)\times[0,T])}^2 \Big)^{\frac12} \\& 
\le \Big(\sum_{j=-\infty}^{+\infty}\|u_{\epsilon'}\|_{L^{\infty}([j,j+1)\times[0,T])}^2 \Big)^{\frac12}
\sup_{j\in \mathbb Z}\|D^{s+\frac{\alpha}2}\mathcal{H}v\|_{L^2([j,j+1)\times[0,T])} \\ & 
\lesssim (1+T)^{\beta}\|u_0\|_{H^s}\sup_{j\in \mathbb Z}\|D^{s+\frac{\alpha}2}\mathcal{H}v\|_{L^2([j,j+1)\times[0,T])},
\end{split}
\end{equation}
where we used \eqref{maintheo3} for the last inequality. Moreover, we argue exactly as in the proof of Proposition \ref{localsmoothing} using that $v$ satisfies equation \eqref{maintheo5} and using the estimates \eqref{convergence3}--\eqref{convergence6} to deduce that
\begin{equation} \label{convergence14} 
\begin{split}
&\sup_{j\in \mathbb Z}\|D^{s+\frac{\alpha}2}\mathcal{H}v\|_{L^2([j,j+1)\times[0,T])} \\ & \lesssim 
\big(1+T+\|\partial_xu_{\epsilon}\|_{L^1_TL^{\infty}_x}+\|\partial_xu_{\epsilon'}\|_{L^1_TL^{\infty}_x}\big)^{\frac12}\|v\|_{L^{\infty}_TH^s_x}+T^{\frac12}\|u_{\epsilon'}\|_{L^{\infty}_TH^s_x}\|\partial_xv\|_{L^2_TL^{\infty}_x} \\ & \quad +T^{\frac12}\big(\|D^{s-1}\partial_x^2u_{\epsilon}\|_{L^2_TL^{\infty}_x}\|v\|_{L^{\infty}_TL^2_x}+ \|\partial_x^2u_{\epsilon}\|_{L^2_TL^{\infty}_x}\|D^{s-1}v\|_{L^{\infty}_TL^2_x}\big).
\end{split}
\end{equation}
Hence, we deduce  from \eqref{convergence10}--\eqref{convergence14} that
\begin{equation} \label{convergence15} 
\begin{split}
\gamma_T^s(v) & \lesssim T^{\kappa_1}\lambda_T^s(v)+T^{\kappa_2}(1+T)^{\beta}\|u_0\|_{H^s}\big(1+T+T^{\frac12}\|u_0\|_{H^s}\big)^{\frac12}(\lambda_T^s(v)+\gamma_T^s(v))\\ & \quad+T^{\kappa_2}\widetilde{g}_{\epsilon,\epsilon'},
\end{split}
\end{equation}
where 
\begin{displaymath}
\begin{split}
\widetilde{g}_{\epsilon,\epsilon'}&=\|D^{s-1+\frac{\alpha}2}\partial_xu_{\epsilon}\|_{L^2_TL^{\infty}_x}\|v\|_{L^{\infty}_TL^2_x}\\ & \quad +(1+T)^{\beta}\|u_0\|_{H^s}\big(\|D^{s-1}\partial_x^2u_{\epsilon}\|_{L^2_TL^{\infty}_x}\|v\|_{L^{\infty}_TL^2_x}+ \|\partial_x^2u_{\epsilon}\|_{L^2_TL^{\infty}_x}\|D^{s-1}v\|_{L^{\infty}_TL^2_x}\big),
\end{split}
\end{displaymath}
so that 
\begin{equation}  \label{convergence16} 
\widetilde{g}_{\epsilon,\epsilon'} \underset{\epsilon,\epsilon' \to 0}{\longrightarrow} 0, 
\end{equation}
due to \eqref{maintheo3}--\eqref{maintheo4bb} and \eqref{maintheo6}.

Therefore, we conclude the proof of Proposition \ref{convergence} gathering \eqref{convergence8}, \eqref{convergence9}, \eqref{convergence15} and \eqref{convergence16}. 
\end{proof}

With Proposition \ref{convergence} at hand, we deduce that $\{u_{\epsilon}\}$ satisfies the Cauchy criterion in $\big(C([0,T_1]:H^s(\mathbb R)),\|\cdot\|_{L^{\infty}_{T_1}H^s_x}\big)$ as $\epsilon$ tends to zero. Therefore, there exists a function $u \in C([0,T_1]:H^s(\mathbb R))$ such that 
\begin{equation} \label{maintheo7}
\|u_{\epsilon}-u\|_{L^{\infty}_{T_1}H^s_x} \underset{\epsilon \to 0}{\longrightarrow} 0 \ .
\end{equation}
 Moreover, we deduce easily from \eqref{maintheo7}, that $u$ is a solution of \eqref{dispBurgers} in the distributional sense and belongs to the class \eqref{maintheo1} (with $T_1$ instead of $T$).

\subsection{Continuity of the flow map}
 Once again, we assume that $0<\alpha<1$ and $\frac32-\frac{3\alpha}8< s \le 3/2$. Fix $u_0 \in H^s(\mathbb R)$. By the existence and uniqueness part, we know that there exists a positive time $T=T(\|u_0\|_{H^s})$ and a unique solution $u \in C([0,T]:H^s(\mathbb R))$ to \eqref{dispBurgers}. Since $T$ is a nonincreasing function of its argument, for any $0<T'<T$, there exists a small ball $B_{\tilde{\delta}}(u_0)$ of $H^s$ centered in $u_0$ and of radius $\tilde{\delta}>0$, \textit{i.e.}
 \begin{displaymath}
 B_{\tilde{\delta}}(u_0)=\big\{v_0 \in H^s(\mathbb R) \ : \ \|v_0-u_0\|_{H^s} < \tilde{\delta} \big\},
 \end{displaymath} 
 such that for each $v_0 \in B_{\tilde{\delta}}(u_0)$, the solution $v$ to \eqref{dispBurgers} emanating from $v_0$ is defined at least on the time interval $[0,T']$.
 
Let $\theta>0$ be given. It suffices to prove that there exists $\delta=\delta(\theta)$ with $0<\delta<\tilde{\delta}$ such that for any initial data $v_0 \in H^s(\mathbb R)$ with $\|u_0-v_0\|_{H^s} < \delta$, the solution $v \in C([0,T'];H^s(\mathbb R))$  emanating from $v_0$ satisfies 
\begin{equation} \label{maintheo8} 
\|u-v\|_{L^{\infty}_{T'}H^s_x} < \theta.
\end{equation}

For any $\epsilon>0$, we normalize the initial data $u_0$ and $v_0$ by defining $u_{0,\epsilon}=\rho_{\epsilon} \ast u_0$ and $v_{0,\epsilon}=\rho_{\epsilon} \ast v_0$ as in the previous subsection and consider the associated smooth solutions $u_{\epsilon}, \ v_{\epsilon} \in C([0,T'];H^{\infty}(\mathbb R))$. Then it follows from the triangle inequality that 
\begin{equation} \label{maintheo9} 
\|u-v\|_{L^{\infty}_{T'}H^s_x} \le \|u-u_{\epsilon}\|_{L^{\infty}_{T'}H^s_x}+\|u_{\epsilon}-v_{\epsilon}\|_{L^{\infty}_{T'}H^s_x}
+\|v-v_{\epsilon}\|_{L^{\infty}_{T'}H^s_x} \ .
\end{equation}
On the one hand, according to \eqref{maintheo7}, we can choose $\epsilon_0$ small enough so that 
\begin{equation} \label{maintheo10} 
\|u-u_{\epsilon_0}\|_{L^{\infty}_{T'}H^s_x}+\|v-v_{\epsilon_0}\|_{L^{\infty}_{T'}H^s_x} < 2\theta/3.
\end{equation}
On the other hand, we get from \eqref{lemma existence1.1} that 
\begin{displaymath} 
\|u_{0,\epsilon_0}-v_{0,\epsilon_0}\|_{H^2} \lesssim \epsilon_0^{-(2-s)}\|u_0-v_0\|_{H^s} \lesssim \epsilon_0^{-(2-s)}\delta.
\end{displaymath}
Therefore, by using the continuity of the flow map for initial data in $H^2(\mathbb R)$ (c.f. Theorem \ref{smoothsol}), we can choose $\delta>0$ small enough such that 
\begin{equation} \label{maintheo11} 
\|u_{\epsilon_0}-v_{\epsilon_0}\|_{L^{\infty}_{T'}H^s_x} < \theta/3.
\end{equation}
Estimate \eqref{maintheo8} is concluded gathering \eqref{maintheo9}--\eqref{maintheo11}.

This concludes the proof of Theorem \ref{maintheo}.

\section{An ill-posedness result}

As in \cite{BKPSV}, one can use the solitary wave solutions to disprove the uniform continuity of the flow map for the Cauchy problem under suitable conditions. More precisely,
we consider again the initial value problem (IVP)
\begin{equation}\label{solwave3}
\begin{cases}
\partial_t u-D^{\alpha}\partial_x u+u\,\partial_x u=0,\;\;x, \,t\in \mathbb R,\\
u(x,0)=u_0(x).
\end{cases}
\end{equation}

\begin{proposition}\label{illposed}
If $ \;1/3\leq \alpha \le 1/2$, then the IVP \eqref{solwave3} 
 is ill-posed in $ H^{s_{\alpha}}(\mathbb R) $ with
$s_{\alpha}=\frac12-\alpha$, in the sense that the time of existence $T$
and the continuous dependence cannot be expressed in terms of the
size of the data in the $H^{s_\alpha}$-norm. More precisely, there exists $c_0>0$ such that for any
$\delta, \,t>0$ small there exist data
$u_1,\,u_2\in\mathcal S(\mathbb R)$ such that
\begin{equation*}
\|u_1\|_{s,2}+\|u_2\|_{s,2}\leq c_0,\;\;\|u_1-u_2\|_{s,2}\leq
\delta,\;\;\|u_1(t)-u_2(t)\|_{s,2}>c_0,
\end{equation*}
where $u_j(\cdot)$ denotes the solution of the IVP \eqref{solwave3} with data
$u_j$,\; $j=1,2$.
\end{proposition}

\begin{remark}
For $\alpha \in [\frac{1}{3},\frac{1}{2}],$ Proposition \ref{illposed} reinforces the result in \cite{MST} which states that the flow map is not $C^2.$
\end{remark}

\begin{proof}
Let $Q_1$ be the solution of the equation
\begin{equation}\label{solwave1}
D^{\alpha}Q+c\,Q-\frac12\,Q^2=0,\quad \alpha\ge 1/3,
\end{equation}
with speed of propagation $c=1$ (see next Section for a justification of existence of such a solution).

Set $\varphi_{\alpha,c}(x)= c^{\alpha}\,Q_1(c x)$ and consider
\begin{equation}\label{solwave2}
u_{\alpha,c}(x,t)=\varphi_{\alpha,c}(x-c^{\alpha}t)= c^{\alpha}\,Q_1(cx-c^{1+\alpha}t)
\end{equation}
solution of the initial value problem \eqref{solwave3} with initial data
$$
u(x,0)=u_{\alpha,c}(x,0)= c^{\alpha}\,Q_1(cx).
$$

We choose  two solutions $u_{\alpha,c_1}, \,u_{\alpha,c_1}$ with $c_1\neq c_2$.  Let $s_{\alpha}=\frac12-\alpha$ be the critical Sobolev
index.

At time $t=0$ we have that
\begin{equation}\label{solwave4}
\begin{split}
\|u_{\alpha,c_1}(\cdot,0)-u_{\alpha,c_2}(\cdot,0)\|^2_{\dot{H^{s_{\alpha}}}}&=\|D^{s_{\alpha}}(\varphi_{\alpha,c_1}-\varphi_{\alpha,c_2})(\cdot)\|_{L^2}^2\\
&=\|D^{s_{\alpha}}\varphi_{\alpha,c_1}(\cdot)\|_{L^2}^2+\|D^{s_{\alpha}}\varphi_{\alpha,c_2}(\cdot)\|_{L^2}^2\\
&\;\;\;\;- 2\langle \varphi_{\alpha,c_1}(\cdot), \varphi_{\alpha,c_2}(\cdot)\rangle_{\dot{H^{s_{\alpha}}}}.
\end{split}
\end{equation}

Observe that  for $t\ge 0$, $\|D^{s_{\alpha}}\varphi_{\alpha,c_j}(\cdot,t)\|_{L^2}^2= \|D^{s_{\alpha}} Q_1\|_{L^2}^2$ for $j=1,2$. In fact,
\begin{equation}\label{solwave5}
\begin{split}
\|D^{s_{\alpha}}\varphi_{\alpha,c_j}(\cdot,t)\|_{L^2}^2&=\int |\xi|^{2s_{\alpha}}c_j^{2\alpha-2}|e^{-2\pi ic_j^{\alpha} t\xi}\,\widehat{Q}_1(\xi/c_j)|^2\,d\xi\\
&=c_j^{2s_{\alpha}+2\alpha-1}\,\|D^{s_{\alpha}} Q_1\|_{L^2}^2=\|D^{s_{\alpha}} Q_1\|_{L^2}^2.
\end{split}
\end{equation}

On the other hand,
\begin{equation}\label{solwave6}
\begin{split}
 \langle \varphi_{\alpha,c_1}(\cdot), \varphi_{\alpha,c_2}(\cdot)\rangle_{\dot{H^{s_{\alpha}}}}
 &=\int D^{s_{\alpha}}\varphi_{\alpha,c_1}(x)\,\overline{ D^{s_{\alpha}}\varphi_{\alpha,c_2}(x)}\,dx\\
 &=\int |\xi|^{2s_{\alpha}}\widehat{\varphi}_{\alpha,c_1}(\xi)\,\overline{\widehat{\varphi}_{\alpha,c_2}(\xi)}\,d\xi\\
 &=(c_1\,c_2)^{(\alpha-1)} \,\int |\xi|^{2s_{\alpha}}\widehat{Q}_1(\xi/c_1)\overline{\widehat{Q}_1(\xi/c_2)}\,d\xi\\
 &=(c_1\,c_2)^{(\alpha-1)}\,c_1^{2s_{\alpha}+1}\,\int |\eta|^{2s_{\alpha}} \,\widehat{Q}_1(\eta)\overline{\widehat{Q}_1(\dfrac{c_1}{c_2}\eta)}\,d\eta\\
 &=\Big(\dfrac{c_1}{c_2}\Big)^{1-\alpha}\,\int |\eta|^{2s_{\alpha}} \widehat{Q}_1(\eta)\,\overline{\widehat{Q}_1(\dfrac{c_1}{c_2}\eta)}\,d\eta.
 \end{split}
 \end{equation}
 
 Now set $\theta=\dfrac{c_1}{c_2}$ such that $\theta\to 1$ then
 \begin{equation}\label{solwave7}
 \langle \varphi_{\alpha,c_1}(\cdot), \varphi_{\alpha,c_2}(\cdot)\rangle_{\dot{H^{s_{\alpha}}}}\to \|D^{s_{\alpha}}Q_1\|_{L^2}^2
  \text{\hskip10pt as\hskip10pt}
 \theta\to 1.
 \end{equation}
 
 Therefore,
 \begin{equation}\label{solwave8}
 \|u_{\alpha,c_1}(\cdot,0)-u_{\alpha,c_2}(\cdot,0)\|^2_{\dot{H^{s_{\alpha}}}}\to 0 \text{\hskip10pt as\hskip10pt} \theta\to 1.
 \end{equation}

Now let $t>0$, as before we only need to check the interaction

\begin{equation}\label{solwave9}
\begin{split}
 \langle u_{\alpha,c_1}(\cdot,t), u_{\alpha,c_2}(\cdot,t)\rangle_{\dot{H^{s_{\alpha}}}}
 &=\int D^{s_{\alpha}}\varphi_{\alpha,c_1}(x-c_1t)\,\overline{ D^{s_{\alpha}}\varphi_{\alpha,c_2}(x-c_2t)}\,dx\\
 &=\int  e^{-2\pi it\xi(c_1^{\alpha}-c_2^{\alpha})}\,|\xi|^{2s_{\alpha}}\widehat{\varphi}_{\alpha,c_1}(\xi)\,\overline{\widehat{\varphi}_{\alpha,c_2}(\xi)}\,d\xi\\
 &=(c_1\,c_2)^{(\alpha-1)} \,\int e^{-2\pi it\xi(c_1^{\alpha}-c_2^{\alpha})}\, |\xi|^{2s_{\alpha}}\widehat{Q}_1(\xi/c_1)\overline{\widehat{Q}_1(\xi/c_2)}\,d\xi\\
 &=\Big(\dfrac{c_1}{c_2}\Big)^{1-\alpha}\,\int e^{-2\pi itc_1\eta(c_1^{\alpha}-c_2^{\alpha})}\, |\eta|^{2s_{\alpha}} \widehat{Q}_1(\eta)\,\overline{\widehat{Q}_1(\dfrac{c_1}{c_2}\eta)}\,d\eta.
 \end{split}
 \end{equation}
 
 Making $c_1^{\alpha}=N+1$ and $c_2^{\alpha}=N$, $N\in \mathbb N$,  and letting $N\to\infty$, the Riemann-Lebesgue lemma implies that
 \begin{equation}\label{solwave10}
 \int e^{-2\pi it\eta(N+1)^{1/\alpha}}\, |\eta|^{2s_{\alpha}} \widehat{Q}_1(\eta)\,\overline{\widehat{Q}_1((\frac{N+1}{N})^{1/\alpha}\,\eta)}\,d\eta
\to 0.
\end{equation}

The result follows.

\end{proof}

\section{Varia and open problems}

\subsection{Solitary waves}

This Subsection is essentially a survey of known results.

A (localized) solitary wave solution of \eqref{dispBurgers} of the form
$u(x,t)=Q_c(x-ct)$  must satisfy the equation
\begin{equation} \label{solitarywave}
D^{\alpha}Q_c + cQ_c -\frac12Q_c^2=0,
\end{equation}
where $c>0$.

One does not expect solitary waves to exist when
$\alpha< \frac13$ since then the Hamiltonian
does not make sense (see a formal argument in \cite{KZ}).
For the sake of completeness, we present here a rigorous proof.

\begin{theorem} \label{NEsolitarywave}
Assume that $0<\alpha \le\frac13$. Then \eqref{solitarywave} does not possesses any nontrivial solution $Q_c$ in the class 
$H^{\frac{\alpha}2}(\mathbb R) \cap L^3(\mathbb R)$\footnote{This implies that the Hamiltonian is well defined.}. 
\end{theorem}

\begin{proof} Fix $0<\alpha<1$ and $c>0$. Let $Q_c$ be a nontrivial solution of \eqref{solitarywave} in the class $H^{\frac{\alpha}2}(\mathbb R) \cap L^3(\mathbb R)$.

On the one hand, we multiply \eqref{solitarywave} by $Q_c$ and integrate over $\mathbb R$ to deduce that 
\begin{equation} \label{NEsolitarywave1}
\int_{\mathbb R}|D^{\frac{\alpha}2}Q_c|^2dx+c\int_{\mathbb R}Q_c^2dx=\frac12\int_{\mathbb R}Q_c^3dx.
\end{equation}

On the other hand, we multiply \eqref{solitarywave} by $xQ_c'$, integrate over $\mathbb R$ and integrate by parts to deduce that 
\begin{equation} \label{NEsolitarywave2}
\int_{\mathbb R}(D^{\alpha}Q_c)xQ_c'dx-\frac{c}2\int_{\mathbb R}Q_c^2=-\frac16\int_{\mathbb R}Q_c^3dx.
\end{equation}
Moreover, we will need the following identity, stated in the proof of Lemma 3 in \cite{KMR},
\begin{equation} \label{NEsolitarywave3}
\int_{\mathbb R}(D^{\alpha}\phi)x\phi'dx=\frac{\alpha-1}2\int_{\mathbb R}|D^{\frac{\alpha}2}\phi|^2dx,
\end{equation}
for all $\phi \in \mathcal{S}(\mathbb R)$. Hence, it follows gathering \eqref{NEsolitarywave2} and \eqref{NEsolitarywave3} that 
\begin{equation} \label{NEsolitarywave4}
(\alpha-1)\int_{\mathbb R}|D^{\frac{\alpha}2}Q_c|^2dx-c\int_{\mathbb R}Q_c^2=-\frac13\int_{\mathbb R}Q_c^3dx.
\end{equation}

Now, we briefly recall the proof of \eqref{NEsolitarywave3} for the sake of completeness. We have by using Plancherel's identity, basic properties of the Fourier transform and integrations by parts that 
\begin{displaymath} 
\begin{split} 
\int_{\mathbb R}(D^{\alpha}\phi)x\phi'dx&=-\int_{\mathbb R}(D^{\alpha}\phi)^{\wedge}(\xi)\frac{d}{d\xi}\overline{(\xi\widehat{\phi}(\xi))}d\xi \\ &=-\int_{\mathbb R}|\xi|^{\alpha}|\widehat{\phi}(\xi)|^2d\xi-\int_{\mathbb R}|\xi|^{\alpha}\xi\widehat{\phi}(\xi)\frac{d}{d\xi}\overline{\widehat{\phi}(\xi)}d\xi \\ & 
=\alpha\int_{\mathbb R}|\xi|^{\alpha}|\widehat{\phi}(\xi)|^2d\xi+\int_{\mathbb R}|\xi|^{\alpha}\xi\frac{d}{d\xi}\widehat{\phi}(\xi)\overline{\widehat{\phi}(\xi)}d\xi \\ & =\alpha\int_{\mathbb R}|D^{\frac{\alpha}2}\phi|^2dx-\int_{\mathbb R}\frac{d}{dx}(x\phi)D^{\alpha}\phi dx \\ & =(\alpha-1)\int_{\mathbb R}|D^{\frac{\alpha}2}\phi|^2dx-\int_{\mathbb R}(D^{\alpha}\phi)x\phi'dx,
\end{split}
\end{displaymath} 
which yields identity \eqref{NEsolitarywave3}. 

Finally, we conclude the proof of Theorem \ref{NEsolitarywave}. In the case $\alpha=\frac13$, we deduce from \eqref{NEsolitarywave1} and \eqref{NEsolitarywave4} that $\int_{\mathbb R}Q_c^2dx=0$ which is absurd. In the case $0<\alpha<\frac13$, we obtain combining \eqref{NEsolitarywave1} and \eqref{NEsolitarywave4} that 
\begin{displaymath} 
\frac12\int_{\mathbb R}|D^{\frac{\alpha}2}Q_c|^2dx=\frac{c}{3\alpha-1}\int_{\mathbb R}Q_c^2dx<0, 
\end{displaymath}
which is also a contradiction. 
 \end{proof}
 
\begin{remark} It is not difficult to see from the proof that Theorem \ref{NEsolitarywave} still holds true in the case $\alpha<0$ if one assumes that $Q_c \in \dot{H}^{\frac{\alpha}2}(\mathbb R) \cap L^3(\mathbb R) \cap L^2(\mathbb R).$ In particular, no solitary waves exist when $\alpha=-\frac{1}{2}.$ Note that this dispersion is that of the Whitham equation for large frequencies. On the other hand, the Whitham equation does possess solitary waves, as proven in \cite{EGW} by using that it behaves as the KdV equation for small frequencies.
 \end{remark}

\begin{remark}
Zaitsev \cite{Z} has proved the existence of localized solitary waves of velocity $0<c<\epsilon^{-2}/3$ of  \eqref{Whit} in the case where $p(\xi)=\frac{\xi^2}{1+\epsilon\xi^2}.$ 

\end{remark}

The existence of finite energy solitary waves  when  $\alpha>\frac{1}{3}$  has been addressed in  \cite{FL}, \cite{F} for the more general class of nonlocal equations in $\R^n$

\begin{equation}\label{frac}
(-\Delta)^{\frac{\alpha}{2}}u+u-u^{p+1}=0.
\end{equation}

In what follows we will consider only the one-dimensional case, $n=1.$

The solitary waves are obtained following Weinstein classical approach by looking for the best constant  $C_{p,\alpha}$in the Gagliardo-Nirenberg inequality

\begin{equation}\label {GN}
\int_\R |u|^{p+2}\leq C_{p,\alpha}\left(\int_\R|D^{\alpha/2}u|^2\right)^{\frac{p}{2\alpha}}\left(\int_\R |u|^2\right)^{\frac{p}{2\alpha}(\alpha-1)+1}, \quad \alpha\geq \frac{p}{p+2}. 
\end{equation}

This amounts to minimize the functional

\begin{equation}\label {MW}
J^{p,\alpha}(u)=\frac{\left(\int_\R|D^{\alpha/2}u|^2\right)^{\frac{p}{2\alpha}}\left(\int_\R |u|^2\right)^{\frac{p}{2\alpha}(\alpha-1)+1}}{\int_\R |u|^{p+2}}.
\end{equation}

\vspace{0.5cm}
In our setting, that is   with $p=1$ and one obtains (see \cite {FL} and the references therein):

\begin{theorem}\label {ESW}
Let $\frac{1}{3}<\alpha <1.$ Then

(i) Existence: There exists a solution $Q\in H^{\frac{\alpha}2}(\R)$ of equation \eqref{solitarywave} such that $Q = Q(|x|) > 0$ is even, positive and strictly decreasing in $|x|$. Moreover, the function $Q\in H^{\frac{\alpha}2}(\R) $ is a minimizer for  $J^{p,\alpha}$.

(ii) Symmetry and Monotonicity: If $Q \in H^{\frac{\alpha}2}(\R)$  is a nontrivial solution of \eqref{solitarywave} with  $Q\geq 0$, then there exists $x_0\in \R$ such that $Q(\cdot -0)$ is an even, positive and strictly decreasing in $|x-x_0|.$

(iii) Regularity and Decay: If  $Q\in H^{\frac{\alpha}2}(\R)$ solves \eqref{solitarywave}, then $Q \in H^{\alpha+1}(\R).$
Moreover, we have the decay estimate
$|Q(x)| + |xQ'(x)| \leq \frac{C}{1+|x|^{1+\alpha}},$
for all $x\in \R$ and  some constant $C >0$.
\end {theorem}

\begin{remark}
Contrary to the case of the KdV equation the solitary wave cannot decay fast (for instance exponentially) because the symbol $i\xi|\xi|^{\alpha}$ of the dispersive term is not smooth at the origin when $\alpha$ is not an even integer.
\end{remark}

Uniqueness issues have been addressed in  \cite{FL}, \cite{F} for the class of nonlocal equations \eqref{frac}. They concern {\it ground states solutions} according to the following definition (see \cite {FL})
\begin{definition}
 Let $Q \in H^{\frac{\alpha}2}(\R)$ be an even and positive solution of \eqref {frac} . If  $$J^{(p,\alpha)}(Q)=\inf \big\{ J^{(p,\alpha)}(u) \ : \ 
 u\in H^{\frac{\alpha}2} (\R)\setminus \lbrace0\rbrace \big\},$$
then we say that $Q$  is a ground state solution. 
\end{definition}
The main  result in \cite{FL} implies in our case ($p=1$) that the ground state is unique when $\alpha >\frac{1}{3}.$

Observe that the uniqueness (up to the trivial symmetries) of the solitary-waves of the Benjamin-Ono solutions has been established in \cite{AT}.

Note that the method of  proof of Theorem \ref{ESW} does not yields any (orbital) stability result. One has to use instead a variant of the Cazenave-Lions method, that is obtain the solitary waves by minimizing the Hamiltonian with fixed $L^2$ norm. This has been done in \cite{ABS} in the case $\alpha =1$ 

If $\frac12 \le \alpha <1,$ one has then to obtain solutions of \eqref{solitarywave} by solving the minimization problem
\begin{equation}\label{var}
\min \lbrace H(u) \ : \  \|u\|_{L^2}=1\rbrace.
\end{equation}
As we previously noticed, results in that direction are obtained in \cite{EGW} where a conditional orbital stability result is given for the original Whitham equation, using in a crucial way that it reduces to the KdV equation in the long wave limit.

\vspace{0.5cm}
On the other hand, it has been established in \cite{KaSt} that the ground state is spectrally stable when $\alpha>\frac{1}{2}.$

No asymptotic results seem to be known (see \cite{KM} for the case of the Benjamin-Ono equation, $\alpha =1$).

\begin{remark}\label{per}
The existence and stability properties of {\it periodic} solitary waves of \eqref{dispBurgers} when $\alpha>\frac{1}{2}$ is studied in \cite {J}.
\end{remark}

\subsection{Long time existence issues}
\vspace{0.5cm}
 An important issue for the rigorous justification of nonlinear dispersive equations as asymptotic models of more complicated systems such as the water waves system, nonlinear Maxwell equations (see for instance \cite{La} in the context of water waves) is the influence of dispersion on the lifespan of solutions to dispersive perturbations of  hyperbolic quasilinear equations or systems which typically arise in water waves theory.
Typically, those systems write

\begin{equation}\label{basic}
\partial_t U+\mathcal{B}U+\epsilon\; \mathcal A(U, \nabla U)+\epsilon \mathcal LU=0,
\end{equation}
where the order $0$ part  $\partial_t U+\mathcal{B}U$ is linear hyperbolic, $ \mathcal L$ being a linear (not necessarily skew-adjoint) dispersive operator and $\epsilon>0$ is a small parameter which measures the (comparable) nonlinear and dispersive effects. Both the linear part and the dispersive part may involves nonlocal terms (see {\it eg} \cite{X}, \cite{SX2}).

Boussinesq systems for surface water waves are a classical example of such systems. \footnote{Note however that the Boussinesq systems \eqref{Bsq} cannot be reduced exactly to the form \eqref {basic} except when $b=c=0.$ Otherwise the presence of a "BBM like " term induces a smoothing effect on one or both nonlinear terms.}

When $ \mathfrak L=0$ one has a quasilinear hyperbolic system and if it is symmetrizable one obtains a lifespan of order $1/\epsilon$ for the solutions of the associated Cauchy problem. 

On the other hand, even when the nonlinear part is symmetrizable and  $ \mathfrak L$ skew-adjoint, the existence on time scales of order $1/\epsilon$ (and actually even the local-well-posedness) is not obvious since the action of the symmetrizer on the dispersive part leads to derivative losses and the energy method does not work in a straightforward way.

A basic question (in particular to justify the validity of \eqref{basic} as an asymptotic model) is thus to prove that the life span of the solution of \eqref{basic} is at least $1/\epsilon$ and to investigate whether or not  this life span is increased by the presence of the dispersive term $\epsilon \mathcal L.$

For scalar (physically relevant) equations the second  question is trivial since they appear most often  as skew-adjoint perturbations of conservation laws of the form 
(after eliminating the transport term by a trivial change of variable)
\begin{equation}\label {triv1}
u_t+\epsilon f(u)_x-\epsilon Lu_x=0,
\end{equation}
for which existence on time scales of order $1/\epsilon$ is trivial.  Actually, whatever the dispersive term $L$ one has the dichotomy: either the solution is global, either its life span has order $0(1/\epsilon),$ as immediately seen by the change of the time variable $\tau=\epsilon t$ which reduces \eqref{triv1} to

\begin{equation}\label {triv2}
u_\tau+ f(u)_x- Lu_x=0,
\end{equation}

This question is not so elementary with a different scaling. A toy model will be again the dispersive Burgers equation written now on the form 
\begin{equation} \label{dispBurgerseps}
\partial_tu+ D^{\alpha}\partial_xu= \epsilon u\partial_xu,\quad u(\cdot,0)=u_0.
\end{equation}
Setting $v=\epsilon u$, this is equivalent to solving
\begin{equation} \label{dispBurgersbis}
\partial_tv+ D^{\alpha}\partial_xv=  v\partial_xv,\quad v(\cdot,0)=\epsilon u_0.
\end{equation}

Our main concern here is to prove the existence of strong solutions to \eqref{dispBurgerseps} defined on time intervals of length  greater than $1/\epsilon$, for small $\epsilon>0$ and 
$-1\leq\alpha\leq 1/2$, $\alpha \neq 0$. Observe that we have hyperbolic blow-up on time $T \sim 1/\epsilon$ in the case $\alpha=0$, which is nothing else than the Burgers equation.

This question is not a simple one, as shows the related example of the {\it Burgers-Hilbert equation}  which corresponds to $\alpha =-1$.
\begin{equation}\label{BH}
u_t+\epsilon uu_x+\mathcal Hu=0, \quad u(\cdot, 0)=u_0,
\end{equation}
where $\mathcal H$ is the Hilbert transform. 

In fact, Hunter and Ifrim \cite{HI} (see a different proof in \cite{HITW}) have shown the rather unexpected result :
\begin{theorem}
 Suppose that $u_0\in  H^2(\R)$. There are constants $k > 0$ and $\epsilon _0 > 0$, depending only on $\|u_0\|_{H^2}$ , such that for every $\epsilon$  with $|\epsilon|\leq  \epsilon_0,$ there exists a solution
$u \in C(I_{\epsilon}; H^2 (\R)) \cap C^1(I_{\epsilon}; H^1 (\R))$
of \eqref{BH}  defined on the time-interval $I_{\epsilon} = \lbrack-k/\epsilon ^2,k/\epsilon ^2\rbrack$ .
\end{theorem}
\vspace{0.3cm}

\begin{remark}
It would be interesting to consider a similar issue for the dispersive Burgers equation \eqref{dispBurgersbis} when $-1<\alpha\leq 1/2$,  $\alpha\neq0.$ One might think of using the 
dispersion, as in the normal form approach (in a different context), see \cite{GMS1, GMS2, GMS3, G} . This will be carried out in a subsequent paper.
\end{remark}

\begin{remark}
Although outside the range of equations studied here, we would like to mention the case $\alpha =-2$ which corresponds to the so-called reduced Ostrowsky equation and for which 
it has been proven in \cite{GP} the existence of global solutions under the following conditions on the initial data $u_0$: $u_0 \in H^3(\mathbb R)$ and $1-3u_0''>0$.  It is interesting 
to note however that an \lq\lq hyperbolic\rq\rq \, blow-up may occur otherwise (\textit{c.f.} Theroem 2 in \cite{GP}).
 \end{remark}

The long time existence issue is specially important to justify rigorously (as asymptotic models)  physically relevant systems such as the Boussinesq systems
\begin{equation} \label{Bsq}
\left\{ \begin{array}{l}   \partial_t\eta+\text{div}\,
\textbf{v}+\epsilon\;\text{div} \, (\eta\textbf{v})+\epsilon( a\;\text{div}\Delta \textbf{v}-b\Delta \eta_t)=0 \\
\partial_t \textbf{v} +\nabla \eta +\epsilon\frac12 \nabla(|\textbf{v}|^2)+\epsilon(c\nabla \Delta
\eta-d\Delta {\bf v}_t)=0 \end{array} \right., \quad (x_1,x_2) \in \mathbb R^2, \ t \in
\mathbb R.
\end{equation}
where $a,b,c,d$ are modelling constants satisfying the constraint $a+b+c+d=\frac{1}{3}$ and ad hoc conditions implying that the well-posedness of  linearized system at the trivial solution $(0,{\bf 0}).$

It has been proven in \cite{SX} (see also \cite{MSZ} and \cite{X, SX2} for another water wave system) that \eqref{Bsq} is well-posed on time scales of order $1/\epsilon$  (with uniform bounds). The method is \lq\lq hyperbolic\rq\rq\, in spirit and works for all the physically admissible Boussinesq systems except the more dispersive one, of \lq\lq KdV-KdV" type
\begin{equation} \label{KdV type}
\left\{ \begin{array}{l}   \partial_t\eta+\text{div}\,
\textbf{v}+\epsilon\;\text{div} \, (\eta\textbf{v})+\epsilon\;\text{div}
\,
\Delta \textbf{v}=0 \\
\partial_t \textbf{v} +\nabla \eta +\epsilon\frac12 \nabla(|\textbf{v}|^2)+\epsilon\;\nabla \Delta
\eta=0 \end{array} \right., \quad (x_1,x_2) \in \mathbb R^2, \ t \in
\mathbb R,
\end{equation}
and for the two-dimensional regularized (BBM) version of the original Boussinesq system.
For \eqref{KdV type} it was proven in \cite{LPS} by using dispersive estimates that the existence time is $O(1/\sqrt \epsilon).$

\begin{remark}
The discrepancy  between the results above can be explained as follows. The proofs using dispersion (that is high frequencies)  do not take into account the algebra (structure)  of the nonlinear terms. They allow initial data in relatively large Sobolev spaces but seem to give only existence times of order $O(1/\sqrt \epsilon).$
 The existence proofs on existence times of order $1/\epsilon$  are of \lq\lq hyperbolic\rq\rq \, nature. They do not take into account the dispersive effects (treated as perturbations). Is it possible to go till $O(1/\epsilon^2),$ or to get global existence? This is plausible in one dimension (the Boussinesq systems should evolves into an uncoupled system of KdV equations see \cite{SW}) but not so clear in two dimensions.\end{remark}

\subsection{Blow-up issues}

\vspace{0.5cm}

We have already mention briefly the possibility of blow-up in finite time for equations like \eqref{dispBurgers} in the case of very weak dispersion ($-1<\alpha<0).$

Actually three different types of blow-up, arising from different phenomena, could occur for \eqref{dispBurgers}.

(i) \lq\lq Hyperbolic\rq\rq \  blow-up, that is blow-up of the gradient, the solution remaining bounded. This is a typical property of scalar conservation laws and it is not likely to occur when $\alpha >0 $ according to the formal argument in \cite{KZ} and the numerical simulations in \cite{KS}.

As already mentioned, a blow-up of this type has been proven for Whitham type equations with a very weak dispersion. This question though is open when $0<\alpha <1,$ one does not even know in this case if  a control on the $L_x^{\infty}$ norm of the solution prevents blow-up as it is the case for the generalized KdV equation, see \cite{ABF}  (this property is of course false for hyperbolic quasilinear equations).



It is interesting to investigate similar issues for weak dispersive perturbations of {\it systems}, for instance for the \lq\lq weakly dispersive\rq\rq \, Boussinesq systems. A good candidate is the system studied by Amick \cite{A} corresponding to
$a=c=b=0$ and $d=\frac{1}{3}$ in \eqref {Bsq} and which was studied  in the 1D case by Amick \cite{A} and Schonbeck \cite{Sc}   as a perturbation of the Saint-Venant system.

Actually they proved  global well-posedness when the  initial data are small, compacted supported, perturbations of constant states. Such initial data leads to gradient blow-up for the underlying Saint-Venant system.

An interesting question is to prove results similar to those of Amick and Schonbeck in the 2D case.

\vspace{0.3cm}
(ii) \lq\lq Dispersive\rq\rq \  blow-up (DBU), or focalization due to the focusing of short or long waves. This phenomenon is typically a linear one (see \cite {BS1, BS2}). Roughly speaking it implies that there exist solutions with smooth, bounded and square integrable initial data that becomes infinite at prescribed points in space-time. One also have solutions starting from decaying, smooth and bounded initial data that can become arbitrary large at prescribed points (see \cite{BS2}).

It is shown in \cite{BS2} that  DBU occurs for the linear fractional Schr\"{o}dinger equations

\begin{equation}\label{fracS}
iu_t+(-\Delta)^{\frac{\alpha}{2}} u=0 \quad \text{in} \;\R^n\times \R.
\end{equation}

 By similar methods (using for instance the asymptotics in \cite{SSS}, one can prove similar results for the {\it linear} equation 
 
 \begin{equation}\label{lindB}
 \partial_t u-D^{\alpha}\partial_x u=0, \quad \alpha >0.
 \end{equation}
For \eqref{lindB} (as for \eqref{fracS} when $\alpha >1$), the DBU is due to the focusing of short waves.
 
 Extending this to the nonlinear equation \eqref{dispBurgers} is an open problem.
 
 \vspace{0.3cm}
(iii) \lq\lq Nonlinear-Dispersive\rq\rq \ blow-up.
This blow-up phenomenum, due to the competition between nonlinearity and dispersion is expected to occur for $L^2$ {\it critical or super-critical} equations such as the generalized Korteweg-de Vries equation (GKdV)

\begin{equation}\label{GKdV}
\partial_t u+u^p\partial_x u+\partial^3_xu=0,
\end{equation}
when $p\geq 4.$
The only known result for GKdV is that of the critical case $p=4$ \cite{MM}\footnote{Note that DBU occurs for all values of $p$ \cite{BS1}.}.  The supercritical case $p>4$ is still open but the numerical simulations in \cite{BDKM} suggest that blow-up occurs in this case too. Recently, Kenig, Martel and Robbiano proved in \cite{KMR} that the same type of blow-up occurs for the critical  equation 
\begin{equation} \label{dispGKdV}
\partial_tu-D^{\alpha}\partial_xu+|u|^{2\alpha}\partial_xu=0
\end{equation}
when $\alpha$ is closed to $2$, \textit{i.e.} near the GKdV equation with critical nonlinearity. Recall that for the dispersive Burgers equation \eqref{dispBurgers}, the critical case corresponds to $\alpha =\frac{1}{2}$  (or $\alpha=\frac12$ for equation \eqref{dispGKdV}).

Things are a bit different for the dispersive Burgers equation \eqref{dispBurgers} equation since in addition to the $L^2$ critical exponent $\alpha =1/2,$ one has the {\it energy critical} exponent $\alpha=1/3$ which has no equivalent for the generalized KdV equations.
As this stage one could conjecture that the Cauchy problem for the dispersive Burgers equation \eqref{dispBurgers} is globally well-posed (in a suitable functional setting) when $\alpha>\frac{1}{2},$ that a blow-up similar to the critical GKdV case, occurs when $\alpha =\frac{1}{2},$ that a \textit{supercritical} blow-up occurs when $\frac{1}{3}\leq \alpha<\frac{1}{2},$ and that a blow-up of a totally nature occurs in the \textit{energy supercritical} case, that is when $0<\alpha<\frac{1}{3}.$This is supported by numerical simulations \cite{KS} but should be difficult to prove.

\vspace{0.5cm}



\subsection{Fractionary BBM equations}

We comment here briefly on the BBM version of the dispersive Burgers equation, namely

\begin{equation}\label{fracBBM}
\partial_tu+\partial_xu+u\partial_xu+D^{\alpha}\partial_tu=0,
\end{equation}
where the operator $D^{\alpha}$ is defined in \eqref{Riesz}.

The case $\alpha=2$ corresponds to the classical BBM equation, $\alpha=1$ to the BBM version of the Benjamin-Ono equation.

For any $\alpha$ the energy

$$E(t)=\int_{\R}(u^2+|D^{\frac{\alpha}{2}}u|^2) dx$$

is formally conserved. By a standard compactness method this implies that the Cauchy problem for \eqref{fracBBM} admits a global weak solution in $L^{\infty}(\R;H^{\frac{\alpha}{2}}(\R))$ for any initial data $u_0=u(\cdot,0)$ in $H^{\frac{\alpha}{2}}(\R).$

One can also use the equivalent form

\begin{equation}\label{fracBBM2}
\partial_t u+\partial_x(I+D^{\alpha})^{-1}\left(u+\frac{u^2}{2}\right)=0,
\end{equation}

which gives the Hamiltonian formulation

$$u_t+J_{\alpha}\nabla_u H(u)=0$$

where the skew-adjoint operator $J_{\alpha}$ is given by $J_{\alpha}=\partial_x(I+D^{\alpha})^{-1}$ and $H(u)=\frac{1}{2}\int_\R (u^2+\frac{1}{3}u^3).$ Note that the Hamiltonian makes  for $u\in H^{\frac{\alpha}{3}}(\R)$ if and only if $\alpha\geq \frac{1}{3}.$

The form \eqref{fracBBM2} shows clearly that the fractionary BBM equation is for $0<\alpha <1$ a kind of "dispersive regularization" of the Burgers equation.
 
\vspace{0.3cm}
We will focus on the case  $0<\alpha <1.$ Actually when $\alpha \geq1,$ \eqref{fracBBM2} is an ODE in the Sobolev space $H^s(\R)$, $s>\frac{1}{2},$ and one obtains by standard arguments (see \cite {Ma, ABS2}) the local well-posedness of the Cauchy problem in $H^s(\R), s>\frac{1}{2}.$ When $\alpha =1$ (the Benjamin-Ono BBM equation), the conservation of energy and an ODE argument as in \cite{S}
 or the Br\' ezis-Gallou\"{e}t inequality (see \cite{ABS}) implies that this local solution is in fact global.

 Things are a bit less simple when $0<\alpha<1$ since \eqref{fracBBM2} is no more an ODE, in any Sobolev space. By  a standard energy method one obtains local well-posedness in $H^s(\R), s>\frac{3}{2}.$  
 
 One can in fact easily improve this result.
 
 \begin{theorem}
 Let $0<\alpha<1.$
 Then the Cauchy problem for \eqref {fracBBM} or \eqref {fracBBM2} is locally well-posed for initial data in $H^r(\R),\; r>r_{\alpha}= \frac{3}{2}-\alpha .$
 \end{theorem}

 \begin{remark}
 It would be interesting to lower the value of $r_{\alpha},$ in particular down to the energy level $r=\frac{\alpha}{2}$, or to prove an ill-posedness result for $r<r_{\alpha}.$
 \end{remark}
 
 \begin{proof}
 We first derive  the suitable energy estimate, that is 
 
\begin{equation}\label{estim}
\frac{d}{dt} \|J^{r}u(\cdot,t)\|_{L^2}^2\leq C \|J^{r}u(\cdot,t)\|_{L^2}^3.
\end{equation}
 
All the following computations can be justified by smoothing the initial data. We set $r=s+\frac{\alpha}{2}.$ One readily obtains
\begin{displaymath}
\frac{1}{2}\frac{d}{dt}\int_{\R}\left(|J^su|^2+|J^{s+\frac{\alpha}{2}}u|^2\right)dx=-\int_{\R}J^s(uu_x)J^sudx.
\end{displaymath}
By the fractional Leibniz rule in \cite {KPV3} (see Lemma \ref{LeibnizRule}) one gets
\begin{equation}
J^s(uu_x)= uJ^su_x+u_xJ^su +R,
\end{equation}
where $R$ is estimated as
\begin{equation}
\|R\|_{L^2}\leq C\|J^{s-\epsilon}u\|_{L^p}\|J^{\epsilon} u_x\|_{L^q},
\end{equation}
for any $0<\epsilon<s$ and $\frac{1}{p}+\frac{1}{q}=\frac{1}{2}.$ Integrating by parts, one has thus
\begin{equation}\label{part}
\int_{\R}J^s(uu_x)J^sudx=\frac{1}{2}\int_{\R}u_x (J^s u)^2dx +\int_{\R}RJ^su dx
\end{equation}
 
Estimating the first integral on the RHS reduces to proving
\begin{equation}\label {est}
\int_{\R}|u_x (J^su)^2|dx\leq C\|u\|_{H^{s+\frac{\alpha}{2}}}^3.
\end{equation}
To do so we use the Sobolev imbedding
$$H^{s+\frac{\alpha}{2}-1}(\R)\hookrightarrow L^p(\R),\quad p\leq p_{\alpha,s}= \frac{2}{3-2s-\alpha}$$
and 
$$H^{\frac{\alpha}{2}}(\R)\hookrightarrow L^q(\R),\quad q\leq q_\alpha = \frac{2}{1-\alpha}.$$
Thus, provided 
$$\frac{1}{p_{\alpha,s}}+\frac{2}{q_\alpha}\leq 1,$$
that is 
$s\geq \frac{3}{2}-\frac{3\alpha}{2},$ one obtains by H\"{o}lder inequality

$$\int_{\R}|u_x (J^su)^2|dx\leq C\|u_x\|_{H^{s+\frac{\alpha}{2}-1}(\R)}\|J^su\|^2_{H^{\frac{\alpha}{2}}(\R)}\leq C\|u\|_{H^{s+\frac{\alpha}{2}}}^3 \, .$$

One now estimate the last integral on the right-hand side of \eqref{part}. We will prove actually that 
\begin{equation}\label{estR}
\Big|\int_{\R}RJ^su dx\Big|\leq \|u\|_{H^{s+\frac{\alpha}{2}}}^3.
\end{equation}
Noticing that $J^su\in H^{\frac{\alpha}{2}}(\R),$ we use the Sobolev imbedding $H^{\frac{\alpha}{2}}(\R)\hookrightarrow L^{\frac{2}{1-\alpha}}(\R)$ to obtain by H\"{o}lder's inequality
$$\Big|\int_{\R}RJ^su dx\Big|\leq C\|R\|_{L^{\frac{2}{1+\alpha}}}\|u\|_{H^r}.$$
By the fractional Leibniz rule, on has for any $0<\epsilon <s,$
\begin{displaymath}
\|R\|_{L^{\frac{2}{1+\alpha}}}\leq C\|J^{s-\epsilon}u\|_{L^{\frac{4}{1+\alpha}}}\|J^\epsilon u_x\|_{L^{\frac{4}{1+\alpha}}}.
\end{displaymath}
Observe that 
$$H^{s+\frac{\alpha}{2}-1-\epsilon}(\R) \hookrightarrow L^{\frac{4}{1+\alpha}}(\R)\quad \text{and}\quad H^{\epsilon+\frac{\alpha}{2}}(\R)\hookrightarrow L^{\frac{4}{1+\alpha}}(\R),$$
provided
$$\frac{1+\alpha}{4}\geq \frac{1}{2}-(s+\frac{\alpha}{2}-1-\epsilon)\quad\text{and}\quad \frac{1+\alpha}{4}\geq \frac{1}{2}-\epsilon-\frac{\alpha}{2}.$$
Choosing such an $\epsilon$ is possible provided
$$1-3\alpha\leq 4s+3\alpha-5,$$
that is when $s\geq \frac{3}{2}-\frac{3 \alpha}{2},$ or $r\geq \frac{3}{2}-\alpha.$ This achieves the proof of  \eqref{estR} and of the local $H^r$ estimate.
 
By classical (compactness) arguments, one gets the existence of a solution $u\in L^{\infty}(0,T^*, H^r(\R)),$  where $T^*=T^*(||u_0||_{H^r})>0.$
 
 We now prove the uniqueness of this local solution. We have of course to take profit of the smoothing effect of the operator $(I+D^{\alpha})^{-1}.$ Heuristically, in the Burgers case ($\alpha =0$), the uniqueness holds when $\|u_x\|_{L^{\infty}_x}$ is controlled. This is replaced here by a control on $\|D^{1-\alpha}u\|_{L^{\infty}_x}$ which is fine since $D^{1-\alpha}u\in H^{r-1+\alpha}(\R),$ and $r-1+\alpha >\frac{1}{2}$ by our choice of $r$.
 
 Let $u$ and $v$ be two solutions and $w=u-v$. One has
\begin{equation}\label{diff}
w_t+\partial_x(I+D^\alpha)^{-1}(w+\frac{1}{2}w(u+v))=0.
\end{equation}
We take the $L^2$  scalar product of \eqref{diff} with $D^sw$ and apply the Leibniz rule and integration by parts to get
\begin{equation}
\begin{split}
 \frac{1}{2}\int_{\R}(|D^sw|^2+|D^{s+\frac{\alpha}{2}}|^2)dx &\leq \int_\R |D^s(u+v)w_x D^sw)|dx\\
 & \quad+\frac{1}{2}\int_\R |(u+v)_x|\;|D^sw|^2dx
 +\int_\R |RD^sw|dx,
\end{split}
\end{equation}
where $\|R\|_{L^2}\leq \|D^{s-\epsilon}(u+v)\|_{L^p}\|D^{\epsilon} w\|_{L^q},$ \quad $\frac{1}{p}+\frac{1}{q}=\frac{1}{2}, \quad \text{for any} \quad 0<\epsilon<s.$ 
As above, the first two integrals on the RHS are majorized by $C||u+v||_{H^r} ||w||_{H^r}^2$ (we recall that $r=s+\frac{\alpha}{2}).$ Similarly, one obtains that 
$$\int_\R |RD^sw|dx\leq \|u+v\|_{H^r}\|w\|_{H^r}^2$$ and we conclude with Gronwall's lemma.
 
The strong continuity in time of the local solution and the continuity of the flow map can be established via the Bona-Smith trick (see for example the proof of Theorem \ref{maintheo}).
 
 \end{proof}
 
 \begin{remark}
 Proving that the existence time of the local solution of
 \begin{equation}
 \partial_tu+\partial_xu+\epsilon u\partial_xu+\epsilon D^{\alpha}\partial_tu=0,
\end{equation}
 is for $\alpha \in (0,1)$ strictly larger than $\mathcal O(\frac{1}{\epsilon})$ is an open question.
\end{remark}
 
\begin{remark}\label{DBU}
 It has been established in \cite{BS2} that the linearization of \eqref{fracBBM} at $0$ displays the dispersive blow-up property if and only if $\alpha\leq 1.$

 Issues concerning a possible \lq\lq nonlinear\rq\rq \, blow-up for \eqref{fracBBM} are not clear and totally open. The numerical simulations in \cite{KS} suggest that a blow up might occur, at least when $0<\alpha<\frac{1}{3}.$
 \end{remark}

\begin{merci}
The Authors were partially  supported by the Brazilian-French program in mathematics. 
They also would like to thank Luc Molinet for suggesting the method of proof of Theorem \ref{maintheo} and Benjamin Texier for suggesting the proof of Remark \ref{IP}.
\end{merci}

\bibliographystyle{amsplain}

\end{document}